\documentclass[12pt]{amsart}
\usepackage{latexsym}
\usepackage{amssymb}
\usepackage{amscd}
\usepackage{mathrsfs}
\usepackage{CJK}
\usepackage[all]{xy}
\usepackage{tikz-cd}

\renewcommand\a{\alpha}
\renewcommand\b{\beta}
\newcommand\g{\gamma}
\renewcommand\d{\delta}
\newcommand\la{\lambda}
\newcommand\z{\zeta}

\renewcommand\th{\theta}

\newcommand\s{\sigma}
\newcommand\x{\chi}

\newcommand\vf{\varphi}

\renewcommand\r{\rho}

\newcommand\w{\omega}

\newcommand\vG{\varGamma}
\newcommand\ve{\varepsilon}

\newcommand\Ql{\bar{\mathbf Q}_l}

\newcommand\BP{\mathbf P}

\newcommand\BF{\mathbf F}

\newcommand\BZ{\mathbf Z}

\newcommand\Bk{\mathbf k}

\newcommand\CQ{\mathcal{Q}}

\newcommand\SE{\mathscr{E}}
\newcommand\SF{\mathscr{F}}

\newcommand\SM{\mathscr{M}}
\newcommand\SN{\mathscr{N}} 

\newcommand\SP{\mathscr{P}}

\newcommand\SH{\mathscr{H}}

\newcommand\SW{\mathscr{W}}

\newcommand\Fg{\mathfrak g}

\newcommand\iv{^{-1}}
\newcommand\wh{\widehat}
\newcommand\wt{\widetilde}
\newcommand\wg{^{\wedge}}

\newcommand\ol{\overline}

\newcommand\lra{\leftrightarrow}

\newcommand\IC{\operatorname{IC}}
\newcommand\Ker{\operatorname{Ker}}
\newcommand\Hom{\operatorname{Hom}}
\newcommand\End{\operatorname{End}}

\newcommand\Lie{\operatorname{Lie}}
\newcommand\Tr{\operatorname{Tr}\,}
\newcommand\ch{\operatorname{ch}}

\newcommand\uni{_{\operatorname{uni}}}

\newcommand\lp{\operatorname{\!\langle\!}}
\newcommand\rp{\operatorname{\!\rangle\!}}

\newcommand\Spin{\operatorname{Spin}}

\newcommand\dw{\dot w}

\newcommand{\isom}{\,\raise2pt\hbox{$\underrightarrow{\sim}$}\,}
\numberwithin{equation}{section}

\newtheorem{thm}{Theorem}[section]
\newtheorem{lem}[thm]{Lemma}

\newtheorem{prop}[thm]{Proposition}

\def \para#1{\par\medskip\textbf{#1}
              \addtocounter{thm}{1}}

\def \remark#1{\par\medskip\noindent
                \textbf{Remark #1}
                \addtocounter{thm}{1}}

\begin{document}
\setlength{\baselineskip}{4.9mm}
\setlength{\abovedisplayskip}{4.5mm}
\setlength{\belowdisplayskip}{4.5mm}
\renewcommand{\theenumi}{\roman{enumi}}
\renewcommand{\labelenumi}{(\theenumi)}
\renewcommand{\thefootnote}{\fnsymbol{footnote}}
\renewcommand{\thefootnote}{\fnsymbol{footnote}}
\parindent=20pt
\medskip
\begin{center}
{\bf Generalized Green functions and unipotent classes \\ 
     for finite reductive groups, III} 
\par
\vspace{1cm}
Toshiaki Shoji
\\
\vspace{0.7cm}
\title{}
\end{center}

\begin{abstract} 
Lusztig's algorithm of computing generalized Green functions 
of reductive groups involves an ambiguity on certain scalars.
In this paper, for reductive groups of classical type with 
arbitrary characteristic, we determine those scalars explicitly, 
and eliminate the ambiguity.  Our results imply that all 
the generalized Green functions of classical type are 
computable.    
\end{abstract}

\maketitle
\pagestyle{myheadings}
\markboth{SHOJI}{GENERALIZED GREEN FUNCTIONS, III}

\bigskip
\medskip

\begin{center}
{\sc Introduction}
\end{center}
\par
\medskip\noindent
This paper is a continuation of [S3, I, II]. 
Let $G$ be a connected reductive group defined over a finite field
$\BF_q$ with Frobenius map $F : G \to G$, and $G\uni$ the unipotent
variety of $G$. Let $\SN_G$ be the set of pairs $(C, \SE)$, where 
$C$ is a unipotent class in $G$, and $\SE$ is a $G$-equivariant simple local 
system on $C$.  Let $L$ be a Levi subgroup of a parabolic subgroup of $G$, 
and $(C_0, \SE_0) \in \SN_L$ be a cuspidal pair on $L$. We assume that 
$L$ is $F$-stable, $C$ is $F$-stable, and $F^*\SE_0 \simeq \SE_0$. 
For a fixed isomorphism $\vf_0: F^*\SE_0 \isom \SE_0$, a generalized Green 
function $Q^G_{L, C_0, \SE_0, \vf_0}$ was defined in [L2, II], which is a $G^F$-invariant
function on $G^F\uni$ with values in $\Ql$.  
Generalized Green functions play a crucial role in 
the representation theory of a finite reductive group $G^F$, 
and it is very important to compute them explicitly since the computation of
irreducible characters of $G^F$ is essentially reduced to the computation 
of those generalized Green functions thanks to the theory of character sheaves [L2].  
\par
Let $\SN_G^F$ be the set of $(C, \SE) \in \SN_G$ such that $C$ is $F$-stable and 
$F^*\SE \simeq \SE$.  
For any $j = (C, \SE) \in \SN^F_G$, we choose an isomorphism $\psi: F^*\SE \isom \SE$.
Note that $\psi$ is unique up to scalar since $\SE$ is simple.  
Let $Y_j$ be the characteristic function 
on $C^F$ defined by $\psi$, and we regard it as a function on $G\uni^F$ by defining 0
outside of $C^F$.   
Then $\{ Y_j \mid j \in \SN_G^F\}$ gives a basis of the $\Ql$-space of $G^F$-invariant 
functions on $G\uni^F$. In [L2, V], Lusztig gave an algorithm of expressing 
$Q^G_{L,C_0, \SE_0, \vf_0}$ as an explicit linear combination of the basis 
$\{ Y_j \mid j \in \SN_G^F\}$, where $\psi =\psi_j: F^*\SE \isom \SE$ is determined from 
the choice of $\vf_0: F^*\SE_0 \isom \SE_0$. (Here $\vf_0$ is chosen under a certain 
mild condition (see 1.2), and $\psi_j$ is determined uniquely from $\vf_0$ (see 1.3).
But note that $\psi_j$ is not directly computable.)
Thus the computation of the 
generalized Green function is reduced to the computation of those $Y_j$. 
Here the function $Y_j$ has a simple structure, up to scalar. 
If we choose a standard isomorphism $\psi_0 : F^*\SE \isom \SE$, the characteristic
function $Y_j^0$ on $C^F$ can be computed explicitly.  
There exists a scalar $\g_j \in \Ql^*$ such that $Y_j = \g_j Y_j^0$ for each $j$.
Thus the determination of generalized Green functions is reduced to the determination 
of those scalars $\g_j$ for each $j \in \SN_G^F$.
This is a non-trivial problem 
\par
The problem of determining $\g_j$ is closely related to 
the problem of finding a $G^F$-class in $C^F$
(note that $C^F$ has finitely many $G^F$-conjugacy classes) which 
has a good property with respect to the Frobenius action on a certain complex 
associated to $(L, C_0,\SE_0)$.      
A priori such a good $G^F$- class in $C^F$ depends on the choice of 
$(L, C_0, \SE_0)$ since $\psi_j$ is determined from $\vf_0: F^*\SE_0 \isom \SE_0$. 
For example, in the case where $L = T$ is a maximal torus, 
$C_0 = \{ 1\}$, and $\SE_0 = \Ql$ is the constant
sheaf, $Q^G_{L, C_0,\SE_0, \vf_0}$ 
is nothing but the original Green function $Q^G_T$.  In that case, 
assuming that $\ch \BF_q = p$ is a good prime, such a good $G^F$-class was
found in [S1], [S2], [BS], called a split class in $C^F$ (essentially, 
it is determined as a $G^F$-class for the adjoint group $G$), 
and the Green functions for exceptional groups were computed explicitly.
\par
Those scalars $\g_j$ were computed explicitly in [S3, I] in the case where 
$G = SL_n$ and $F$ is of split type (for any $p$), $F$ is of non-split type 
($p$ : large enough), and in [S3, II] in the case where $G = Sp_N$ or $SO_N$ 
for any $p$.      
The main discovery there is that the split class for 
$(T, \{ 1\}, \Ql)$ behaves well for any choice of $(L, C_0, \SE_0)$,
namely, there exists a $G^F$-class in $C^F$ which has a good property for 
any choice of $(L, C_0, \SE_0)$, and the image of this class to the adjoint 
group coincides with the split class for $(T, \{ 1\}, \Ql)$ mentioned above.   
Such a class is called a split class in $C^F$.  
\par
Since the split class is conjecturally compatible with the isogeny, 
the existence of the split class is reduced to the case of 
simply-connected groups.
$Sp_N$ or $SL_n$ are simply-connected, but $SO_N$ is not. In that case, 
we need to consider the spin group $\Spin_N$.  In this paper, we prove that there exists
a split element in $\Spin_N$ such that its image in $SO_N$ is a split element 
in the sense of [S3, II]. 
We also prove the existence of the split elements in the case where 
$G = SL_n$ and $F$ is of non-split type, for an arbitrary $p > 0$.  
We determine the scalars $\g_j$ in those cases.  Thus the generalized Green 
functions can be computed in those cases.  Since the determination of 
generalized Green functions is reduced to the case of simply-connected, 
simple groups, our results imply that all the generalized Green functions 
are computable for any reductive group of classical type. 
\par
The main ingredients for the proof is the Frobenius version of the
restriction theorem  (1.8.1) proved in [S3, II].  This is especially useful 
in the case of $SL_n$ of non-split type. In [S3, I], the existence of split 
elements was proved by the aid of the graded Hecke algebras,  but in that case 
we need to assume that $p$ is large enough for applying Jacobson-Morozov theorem 
in $p > 0$.  The present proof using the restriction theorem is much simpler, and 
works well for any $p$.   

\section{Generalized Green functions}

\para{1.1.}
Let $G$ be a connected reductive group defined over a finite field 
$\BF_q$ with Frobenius map $F$. Let $\Bk$ be the algebraic closure of 
$\BF_q$ with $\ch \Bk = p$. 
Let $\SM_G$ be the set of triples $(L, C_0, \SE_0)$, up to $G$-conjugacy, 
where $L$ is a Levi
subgroup of a parabolic subgroup of $G$, and $\SE_0$ is a cuspidal local 
system on a unipotent class $C_0$ of $L$. As in (1.2.2) in [S3, I], one can 
define a semisimple perverse sheaf $K$ on $G$ associated to 
a triple $(L, C_0, \SE_0) \in \SM_G$.  Let $\SW _L = N_G(L)/L$.  It is known 
that $\SW_L$ is a Coxeter group, and $\End K \simeq \Ql[\SW_L]$, $K$ is 
decomposed as
\begin{equation*}
\tag{1.1.1}
K \simeq \bigoplus_{E \in \SW_L\wg}E \otimes K_E,
\end{equation*}
where $K_E \simeq \Hom (E, K)$ is a simple perverse sheaf associated to $E \in \SW_L\wg$. 
\par
Let $G\uni$ be the unipotent variety of $G$, and 
$\SN_G$ the set of $(C,\SE)$, where $C$ is a unipotent class of $G$, and 
$\SE$ is an irreducible local system on $C$.  
Let $Z_L$ be the center of $L$, and put $d = \dim Z_L$.  It is known that
$K[-d]|_{G\uni}$ is a semisimple perverse sheaf on $G\uni$, equipped with $\SW_L$-action, 
and is decomposed as
\begin{equation*}
\tag{1.1.2}
K[-d]|_{G\uni} \simeq \bigoplus_{(C,\SE) \in \SN_G}
                V_{(C,\SE)} \otimes \IC(\ol C, \SE)[\dim C],
\end{equation*}
where $V_{(C, \SE)}$ is an irreducible $\SW_L$-module if it is non-zero.
It follows from (1.1.1) 
and (1.1.2) that there exists a map $\SW_L\wg \to \SN_G$, 
$E \mapsto (C,\SE)$ such that 
\begin{equation*}
\tag{1.1.3}
K_E|_{G\uni} \simeq \IC(\ol C, \SE)[\dim C + \dim Z_L]
\end{equation*}
and that $V_{(C,\SE)} \simeq E$. The generalized Springer correspondence shows that 
this map gives a bijection 
\begin{equation*}
\tag{1.1.4}
\bigsqcup_{(L, C_0, \SE_0) \in \SM_G}\SW_L\wg \to \SN_G.
\end{equation*}

\para{1.2.}
$F$ acts naturally on the set $\SN_G$ and $\SM_G$ by 
$(C, \SE) \mapsto (F\iv(C), F^*\SE)$, 
$(L, C_0, \SE_0) \mapsto (F\iv(L), F\iv(C_0), F^*\SE_0)$. 
The map in (1.1.4) is compatible with this $F$-action. 
We choose $(L, C_0, \SE_0) \in \SM_G^F$. Here we may assume that 
$L$ is an $F$-stable Levi subgroup of an $F$-stable parabolic subgroup $P$ of $G$,
and that $F(C_0) = C_0, F^*\SE_0 \simeq \SE_0$.  
Since $\SE_0$ is a simple local system, the isomorphism $\vf_0: F^*\SE_0 \isom \SE_0$
is unique up to scalar.
We choose $\vf_0$ by the condition that 
the induced map on the stalk of $\SE_0$ at any point in $C_0^F$ is of finite order.  
$\vf_0$ induces a natural isomorphism $\vf : F^*K \isom K$. 
We consider the characteristic function $\x_{K,\vf}$ of $K$.  The restriction 
of $\x_{K,\vf}$ on $G^F\uni$ is the so-called generalized Green function 
$Q^G_{L, C_0, \SE_0, \vf_0}$ (cf. [L2, II]), 
which is a $G^F$-invariant function on $G^F\uni$ with values in $\Ql$.

\para{1.3.}
Let $(L,C_0, \SE_0) \in \SM_G^F$ be as in 1.2. Then $F$ acts naturally on 
$\SW_L$, which induces a Coxeter group automorphism $\s$ of order $c$.  We consider 
the semidirect product $\wt \SW_L = \lp \s \rp \ltimes \SW_L$.  If an irreducible 
representation $E$ of $\SW_L$ is $F$-stable, namely, $F(E) \simeq E$ as $\SW_L$-modules, 
$E$ can be extended to a $\wt\SW_L$-module
in $c$ different ways.  For each $E \in (\SW_L\wg)^F$, we choose an irreducible representation 
$\wt E$ of $\wt\SW_L$, which is the preferred extension of $E$ 
in the sense of [L2, IV, (17.2)].   
Let $\s_E : E \isom E$ be the action of $\s$ on $\wt E$.  
If $E$ is $F$-stable, then $E \otimes K_E$ is $F$-stable in the decomposition 
of $K$ in (1.1.1), and 
there exists a unique isomorphism $\vf_E : F^*K_E \isom K_E$ 
such that $\s_E \otimes \vf_E : F^*(E \otimes K_E) \isom  E \otimes K_E$ 
is the restriction of $\vf : F^*K \isom K$. 
\par
Let $(C,\SE) \in \SN_G^F$ be any pair.  By the generalized Springer correspondence, 
there exists a unique triple $(L, C_0, \SE_0) \in \SM_G^F$ such that $(C, \SE)$ corresponds to
$E \in (\SW_L\wg)^F$ under (1.1.4).  Put 
\begin{equation*}
\tag{1.3.1}
\begin{aligned}
a_0 &= -\dim Z_L - \dim C, \\ 
r &= \dim G - \dim L + \dim (C_0 \times Z_L)
\end{aligned}
\end{equation*}
Then we have
\begin{equation*}
\tag{1.3.2}
a_0 + r = (\dim G - \dim C) - (\dim L - \dim C_0).
\end{equation*}
By (1.1.3), $\SH^{a_0}K_E|_C = \SE$.
We define an isomorphism $\psi : F^*\SE \isom \SE$ so that 
$q^{(a_0+r)/2}\psi$ corresponds to the map 
$F^*(\SH^{a_0}K_E) \isom \SH^{a_0}K_E$ induced from $\vf_E$. 
For each $j = (C,\SE) \in \SN_G^F$, we define a function $Y_j$ on $G^F\uni$
by 
\begin{equation*}
\tag{1.3.3}
Y_j(g) = \begin{cases}
           \Tr(\psi, \SE_g)  &\quad\text{ if $g \in C^F$, } \\
            0                &\quad\text{ otherwise. }
         \end{cases}
\end{equation*}

\para{1.4.}
The set $\{ Y_j \mid j \in \SN_G^F\}$ gives a basis of the $\Ql$-space of 
$G^F$-invariant functions of $G^F\uni$. 
In [L2, V], Lusztig gave a general algorithm of computing the generalized Green 
function as an explicit linear combination of various $Y_j$. Thus the determination 
of the generalized Green functions is reduced to the determination of $Y_j$.  
Since $\SE$ is a simple local system, the isomorphism $\psi : F^*\SE \isom \SE$
is unique up to scalar. 
The typical isomorphism  $\psi_0: F^*\SE \isom \SE$ among them 
is given as follows:
choose $u \in C^F$, and put $A_G(u) = Z_G(u)/Z_G^0(u)$. Then $F$ acts naturally on 
$A_G(u)$, and the set of $G$-equivariant simple local systems $\SE'$ on $C$ such that 
$F^*\SE' \simeq \SE'$ is in bijection with the set of $F$-stable irreducible 
representations of $A_G(u)$. The correspondence is given so that 
the stalk $\SE'_u$ has the structure of the dual of the corresponding irreducible $A_G(u)$-module.
Let $\r$ be the $F$-stable irreducible representation of $A_G(u)$ corresponding to $\SE$. 
Let $\tau$ be the restriction of $F$ on $A_G(u)$, and 
consider the semidirect product $\wt A_G(u) = \lp\tau\rp \ltimes A_G(u)$. 
Then $\r$ can be extended to an irreducible representation of $\wt A_G(u)$. We choose
an extension $\wt\r$ of $\r$.  Here we consider $\SE'_u$ as an irreducible 
$\wt A_G(u)$-module corresponding to $\wt \r^*$, where $\wt\r^*$ is the dual of $\wt\r$. 
\par
We define $\psi_0: F^*\SE \isom \SE$ by the condition that the induced map on $\SE_u$ 
coincides with the action of $\tau$, 
and define a function $Y_j^0 : G^F\uni \to \Ql$
in a similar way as $Y_j$, but by replacing $\psi$ by $\psi_0$. 
Note that the set of $G^F$-conjugacy classes in $C^F$ is 
in bijection with the set of $A_G(u)$-conjugacy classes in the coset 
$A_G(u)\tau \subset \wt A_G(u)$. We denote by
$u_a$ the $G^F$-class corresponding to the class $a\tau \in A_G(u)\tau$. 
Then $Y_j^0$ is explicitly given as follows;
\begin{equation*}
\tag{1.4.1}
Y_j^0(g) = \begin{cases}
             \Tr(a\tau, \wt\r^*) &\quad\text{ if $g$ is $G^F$-conjugate to $u_a$, } \\
             0          &\quad\text{ if $g \notin C^F$.}
           \end{cases}
\end{equation*}  

\par
Since $\psi$ can be written as $\psi = \g\psi_0$ for some $\g \in \Ql^*$, 
we have $Y_j = \g Y_j^0$.  Thus the determination of $Y_j$ is reduced 
to that of $\g$.  
We choose $u_0 \in C_0^F$, and consider 
$\wt A_L(u_0) = \lp\tau_0\rp \ltimes A_L(u_0)$.
Let $\r_0$ be the irreducible representation of $A_L(u_0)$ corresponding to $\SE_0$, 
and $\wt \r_0$ its extension to $\wt A_L(u_0)$.  We can define $\vf_0: F^*\SE_0 \isom \SE_0$
by the condition that the map on $(\SE_0)_{u_0}$ induced from $\vf_0$ 
coincides with the action of $\tau_0$ on $\wt\r_0$.  Thus $\g$ depends on the choice of 
$u_0, \wt\r_0, u, \wt\r$, and we denote it by $\g = \g(u_0, \wt\r_0, u, \wt\r)$. 
In this paper, we shall prove the following result.

\begin{thm} 
Assume that $G$ is a connected reductive group of classical type.
For each $F$-stable unipotent class $C$ of $G$, 
there exists $u^{\bullet} \in C^F$ satisfying the following property; 
for a triple $(L, C_0, \SE_0) \in \SM_G^F$, 
take $u_0^{\bullet} \in C_0^F$ as above (applied for $L$), and choose 
an extension $\wt\r_0 \in \wt A_L(u_0^{\bullet})\wg$, where $\r_0$ corresponds to $\SE_0$. 
For $(C, \SE) \in \SN_G^F$,  take $u^{\bullet} \in C^F$, and let
$\r \in A_G(u^{\bullet})\wg$ be such that $(C, \SE) \lra (u^{\bullet}, \r)$.
Then for $(C,\SE)$ belonging to the series $(L, C_0, \SE_0)$, there exists 
an extension $\wt\r$ of $\r$ such that $\g(u_0^{\bullet}, \wt\r_0, u^{\bullet}, \wt\r) = 1$.  
\par
In particular, for such a choice of 
$\vf_0: F^*\SE_0 \isom \SE_0$, $Y_j$ coincides with $Y_j^0$.  Thus the generalized Green 
functions $Q^G_{L, C_0, \SE_0, \vf_0}$ are explicitly computable for $G$.     
\end{thm}

\para{1.6.}
$u^{\bullet} \in C^F$ in the theorem is called a split element, and the $G^F$-conjugacy class 
containing $u^{\bullet}$ is called the split class in $C^F$.  
Note that the split class is independent from the choice of $(L, C_0, \SE_0)$. 
Also note that the split class
is compatible with the isogeny. 
(In the case where $G$ is an adjoint group, the split class is uniquely 
determined for a given $C^F$. In general they are not unique, for example, 
see 3.2 in the case of spin groups, and 4.7 for special linear groups.) 

\par
The proof of the theorem is reduced to the case where $G$ is simply-connected, and simple.
The existence of split elements (and the existence of $\wt\r$ for the non-split case)
was proved in [S3, I] for the case where $(G, F)$ is $SL_n$ of split type for any $p > 0$, 
or $SL_n$ of non-split type for sufficiently large $p$. 
In [S3, II], it was proved for $G = Sp_N$ or $SO_N$ for any $p > 0$.  
Thus it remains to verify it for the case where $G$ is the spin group $\Spin_N$, 
or $SL_n$ of non-split type for arbitrary $p$.  In this paper, we discuss these two cases. 

\para{1.7.}
Let $(L, C_0, \SE_0) \in \SM_G$, and let $P$ be a parabolic subgroup of $G$ 
whose Levi subgroup is $L$.  We consider another parabolic subgroup $Q \supset P$ of $G$ 
with the Levi subgroup $M \supset L$. Then $\SW_L' = N_M(L)/L$ is in a natural way 
a subgroup of $\SW_L$.
Let $u \in G,u' \in M$ be unipotent elements.  We put 
\begin{equation*}
\tag{1.7.1}
Y_{u,u'} = \{x \in G \mid x\iv ux \in u'U_Q\}. 
\end{equation*}
Then $Z_G(u) \times Z_M(u')U_Q$ acts on $Y_{u,u'}$ by 
$(g, m) : x \mapsto gxm\iv$ ($g \in Z_G(u), m \in Z_M(u')U_Q, x \in Y_{u,u'}$.  
Put 
\begin{equation*}
\tag{1.7.2}
d_{u,u'} = \frac{1}{2}(\dim  Z_G(u) + \dim Z_M(u')) + \dim U_Q.
\end{equation*}
It is known that $\dim Y_{u,u'} \le d_{u,u'}$.  We denote by $X_{u,u'}$ 
the set of irreducible components of $Y_{u,u'}$ of dimension $d_{u,u'}$. 
Then $A_G(u) \times A_M(u')$ acts on $X_{u,u'}$.  We denote by $\ve_{u,u'}$ 
the corresponding permutation representation of $A_G(u) \times A_M(u')$. 
\par
We denote by $\th^G$ the inverse map of (1.1.4), i.e., 
\begin{equation*}
\tag{1.7.3}
\th^G : \SN_G \to \bigsqcup_{(L,C_0, \SE_0)\in \SM_G}\SW_L\wg.
\end{equation*}
For each pair $(C, \SE) \in \SN_G$, we choose $(u, \r)$, where $u \in C$, 
and $\r \in A_G(u)\wg$ the irreducible representation  corresponding to $\SE$. 
We often write it as $(u, \r) \lra (C, \SE)$. 
We denote by $\th^G_{u,\r} \in \SW_L\wg$ the representation $E$ corresponding to
$(C,\SE)$ under $\th^G$. The generalized Springer correspondence $\th^M$ for $M$ 
is defined similarly to (1.7.3).
For each $(C',\SE') \in \SN_M$, we choose $(u',\r) \lra (C',\SE')$.
Then $\th^M_{u',\r'} \in (\SW_L')\wg$.    
The restriction theorem ([LS, 0.4, (4)] shows that 
\begin{equation*}
\tag{1.7.4}
\lp \r\otimes \r'^*, \ve_{u,u'}\rp_{A_G(u) \times A_M(u')} = 
                     \lp \th^G_{u, \r}|_{\SW_L'}, \th^M_{u',\r'}\rp_{\SW_L'}
\end{equation*}
if $(u,\r)$ and $(u',\r')$ belong to the same class $(L,C_0, \SE_0)$, where
$\r'^*$ is the dual representation of $\r'$.  
The left hand side is equal to zero if $(u,\r)$ and $(u',\r')$ belong to 
a different class. Moreover, every irreducible representation of 
$A_G(u) \times A_M(u')$ which occurs in $\ve_{u,u'}$ is obtained in this way. 

\para{1.8.}
We assume that $L \subset P, M \subset Q$ are both $F$-stable. 
We consider $Y_{u,u'}$ in (1.7.1), and assume that $u, u'$ are $F$-stable. 
Then $F$ acts naturally on $Y_{u,u'}$, and acts on $X_{u,u'}$ as a permutation
of irreducible components. 
We have an action of $F$ on $\ve_{u,u'}$. 
Put $A(u,u') = A_G(u) \times A_M(u')$, then $F$ acts diagonally on $A(u,u')$.
We define $\wt A(u,u') = \lp \tau_* \rp \ltimes A(u,u')$, where $\lp \tau_* \rp$ 
is the infinite cyclic group generated by $\tau_*$, and $\tau_*$ acts on 
$A(u,u')$ via $F$. 
Now $\ve_{u,u'}$ turns out to be 
an $\wt A(u,u')$-module, which we denote by $\wt\ve_{u,u'}$. 
In turn, if we choose extensions $\wt\r, \wt\r'$ of $\r, \r'$, respectively, 
they define an extension $\wt{\r \otimes \r'^*}$ of $\r \otimes \r'^*$ 
to irreducible $\wt A(u,u')$-module, where $\tau_*$ acts on 
$\r \otimes \r'^*$ as $\tau \otimes \tau'$ for 
$\wt A_G(u) = \lp \tau\rp \ltimes A_G(u), \wt A_M(u') = \lp \tau'\rp \ltimes A_M(u')$. 
\par
On the other hand, for each $(C, \SE) \in \SN_G^F$ corresponding to $(u, \r)$, 
we consider the map 
$\psi_0 : F^*\SE \isom \SE$ constructed in 1.4, which we denote by $\psi_{u,\r}$.
$\psi_{u,\r}$ can be extended to 
$\wt\psi_{u,\r}: F^*\IC(\ol C, \SE) \isom \IC(\ol C, \SE)$.  By comparing 
this with  the isomorphism $\vf : F^*K \isom K$ given in 1.2, we obtain an isomorphism 
$\s_{(u,\r)}: V_{(u,\r)} \isom V_{(u,\r)}$ such that 
the restriction of $\vf$ on $V_{(u,\r)}\otimes F^*\IC(\ol C,\SE)[\dim C])$
coincides with $\s_{(u,\r)} \otimes \wt\psi_{(u,\r)}$. (Here we write 
$V_{(C,\SE)}$ as $V_{(u,\r)}$). 
Applying this to the reductive group $M$, and $(C',\SE') \in \SN_M^F$, 
we obtain the map $\s_{(u',\r')} : V_{(u',\r')} \isom V_{(u',\r')}$, 
where $(u',\r') \lra (C', \SE')$. Here $V_{(u,\r)}$ is an irreducible $\SW_L$-module, 
and $V_{(u',\r')}$ is an irreducible $\SW_L'$-module.   
We can decompose $V_{(u,\r)}$ as 
\begin{equation*}
V_{(u,\r)} = \bigoplus_{V_{(u',\r')} \in (\SW'_L)\wg}M_{\r, \r'}\otimes V_{(u',\r')},
\end{equation*}
where $M_{\r, \r'} = \Hom_{\SW_L'} (V_{(u',\r')}, V_{(u,\r)}|_{\SW_L'})$ is the 
multiplicity space for $V_{(u',\r')}$.  
Thus we can define an isomorphism $\s_{\r,\r'}: M_{\r,\r'} \isom M_{\r,\r'}$
such that the restriction of $\s_{(u,\r)}$ on $M_{\r,\r'}\otimes V_{(u',\r')}$
coincides with $\s_{\r,\r'}\otimes\s_{(u',\r')}$.    
\par
A variant of the restriction theorem (1.7.4)
was proved in [S3, II, Cor.1.9].  Under the notation above, the following formula holds.
\begin{equation*}
\tag{1.8.1}
\Tr(\s_{\r,\r'}, M_{\r,\r'}) = q^{-(\dim C - \dim C')/2 + \dim U_Q}
                 \lp \wt\ve_{u,u'}, \wt{\r \otimes \r'^*}\rp_{A(u,u')\tau_*}, 
\end{equation*}
where for representations $V_1, V_2$ of $\wt A(u,u')$, 
$\lp V_1 ,V_2\rp_{A(u,u')\tau_*}$ is defined as
\begin{equation*}
\lp V_1, V_2\rp_{A(u,u')\tau_*} = |A(u,u')|\iv \sum_{a \in A(u,u')}
                                  \Tr(a\tau_*,V_1)\Tr((a\tau_*)\iv, V_2). 
\end{equation*} 

We note that the following fact ([S3, II, Lemma 1.11]) is useful 
for the determination of $\g$, combined with (1.8.1).

\par\medskip\noindent
(1.8.2) \  Suppose that $q^{-(a_0 + r)/2}\s_{(u,\r)}$ makes $\SW_L$-module
$V_{(u,\r)}$ the preferred extension to $\wt\SW_L$.  Then we have $\g = 1$. 
In particular, in the case where $F$ acts trivially on $\SW_L$, if 
$\s_{(u,\r)}$ is $q^{(a_0 + r)/2}$ times identity,  then $\g = 1$. 

\para{1.9.}
Let $Z_G$ be the center of $G$. For each $\x \in Z_G\wg$, we denote by 
$\SN_{\x}$ the set of $(C, \SE) \in \SN_G$ such that 
$Z_G$ acts on $\SE$ according to $\x$.  We denote by 
$A_G(u)\wg_{\x}$ the set of $\r \in A_G(u)\wg$ such that $Z_G$ acts on 
$\r$ according to $\x$.  If $(C, \SE) \lra (u,\r)$, then 
$(C, \SE) \in \SN_{\x}$ if and only if $\r \in A_G(u)\wg_{\x}$.   
Also we denote by $\SM_{\x}$ the set of $(L, C_0, \SE_0) \in \SM_G$
such that $Z_G$ acts on $\SE_0$ according to $\x$. 
We have a partition 
\begin{equation*}
\SN_G = \bigsqcup_{\x \in Z_G\wg}\SN_{\x}, \quad
A_G(u)\wg = \bigsqcup_{\x \in Z_G\wg}A_G(u)\wg_{\x},  \quad
\SM_G = \bigsqcup_{\x \in Z_G\wg}\SM_{\x}. 
\end{equation*}

Then the generalized Springer correspondence $\th^G$ is compatible 
with the action of $Z_G$, namely for each $\x \in Z_G\wg$, $\th^G$ restricts 
to a bijection 
\begin{equation*}
\tag{1.9.1}
\th^{\x} : \SN_{\x} \to \bigsqcup_{(L,C_0,\SE_0) \in \SM_{\x}}\SW_L\wg.
\end{equation*}

\par\bigskip
\section{ Preliminaries on spin groups}

\para{2.1.}
First we review the definition of spin groups. Let $\Bk$ be an algebraically 
closed field with $\ch \Bk \ne 2$.  Let $V$ be a vector space over $\Bk$ 
with dimension $N \ge 3$, endowed with a non-degenerate symmetric bilinear form 
$(\ ,\ )$ on $V$. 
Let $C(V)$ be the Clifford algebra associated to the form $(\ ,\ )$. It is provided 
with an embedding $V \subset C(V)$, and the product $v\cdot v'$ for $v, v' \in V$
satisfies the relation $v\cdot v' + v'\cdot v = 2(v,v')$. 
Let $C^+(V)$ (resp. $C^-(V)$) be the subspace of $C(V)$ spanned by 
products of an even number (resp. an odd number) of elements in
$V$. Then $C(V) = C^+(V) \oplus C^-(V)$, and $C^+(V)$ is 
the subalgebra of $C(V)$. The spin group $\Spin (V)$ is the subgroup of the units of $C^+(V)$ 
consisting of all products $v_1\cdots v_a$ with $a$ : even, where $v_i \in V$ 
satisfy $(v_i,v_i) = 1$. This is a closed subgroup of units in $C^+(V)$. 
If $x \in \Spin(V)$, the map $v \mapsto xvx\iv$ leaves $V$ invariant, and 
defines an element $\b(x)$ of $SO(V)$; if $x = v_1\cdots v_a$ as above, 
then $\b(x) = \b(v_1)\cdots \b(v_a)$, where $\b(v_i)(v) = -v + 2(v,v_i)v_i$ 
is $(-1)$ times the reflection with respect to $v_i$. 
Thus we have a homomorphism $\b : \Spin(V) \to SO(V)$. If we fix a basis, and 
identify $SO(V)$ with $SO_N$, then we also write $\Spin(V)$ as $\Spin_N$. 
$\b$ is the simply-connected covering of $SO(V)$.
We denote $\Spin(V)$ as $G$ and $SO(V)$ as $\ol G$. 
\par
Let $Z$ be the center of $G$.  If $N$ is odd, $Z$ has order 2, it is generated by 
$\ve = (-1)$ times the unit element in $C(V)$.  If $N$ is even, $Z$ has order 4, 
it is generated by $\ve$ and $\w = v_1\cdots v_N$, where $v_1, \dots, v_N$ are 
orthonormal basis of $V$. We have $\w^2 = \ve^{N/2}$.  Hence $Z$ is cyclic of 
order 4 generated by $\w$ if $N \equiv 2\pmod 4$, and 
$Z \simeq \lp \w\rp \times \lp \ve \rp$ is a product of two
cyclic groups of order 2 if $N \equiv 0\pmod 4$.       
Put $Z_0 = \Ker \b$.  Then $Z_0 = \{ 1, \ve\}$ is of order 2. 

\para{2.2.}
We consider the $\BF_q$-structure of $G$. Assume that $V$ is defined 
over $\BF_q$ with Frobenius map $F: V \to V$, and that the form $(\ ,\ )$
is compatible with $F$-action, namely $F(v,w) = (F(v), F(w))$. Then 
$C(V)$ has a natural $\BF_q$-structure, which induces an $\BF_q$-structure of $G$. 
We have the corresponding Frobenius map $F: G \to G$. $\ol G$ has also a natural
$\BF_q$-structure with Frobenius map $F$, and the map $\b : G \to \ol G$ is 
$F$-equivariant. 
Here $F$ acts trivially on $Z_0 = \Ker \b$.  Thus $F$ acts trivially on $Z$ 
if $N$ is odd. Assume that $N \equiv 0\pmod 4$. 
The quadratic form on $V$ associated to $(\ ,\ )$ is given for $x = \sum_ix_ie_i$ 
\begin{equation*}
(x,x) = \begin{cases}
           x_1^2 + \cdots + x_N^2  &\quad\text{ split case,} \\
           x_1^2 + \cdots + x_{N-1}^2 + \d x^2_N
                   &\quad\text{ non-split case,} 
        \end{cases}
\end{equation*}
where $e_1, \dots, e_N$ is a basis of $V$ such that $F(e_i) = e_i$, 
and $\d \in \BF_q - \BF_q^2$.
Thus in the split case, $v_i = e_i$ gives an orthonormal basis of $V$, 
stable by $F$.  This implies that $\w = v_1\cdots v_N$ is $F$-stable. 
In the non-split case, $v_1, \dots, v_N$ gives an orthonormal basis of $V$, 
where $v_i = e_i$ for $i = 1, \dots, N-1$, and $v_N = \d^{1/2}e_N$ 
with $\d^{1/2} \notin \BF_q$.  Here we have $F(\d^{1/2}) = -\d^{1/2}$, and
so $F(v_N) = -v_N$.  This implies that $F(\w) = -\w$.  Summing up the above 
argument, we have
\par\medskip\noindent
(2.2.1) \ Assume that $N$ is odd, or $N \equiv 2\pmod 4 $.  Then $F$ acts 
trivially on $Z$.  Assume that $N \equiv 0\pmod 4$.  If $(G,F)$ is of split type, 
then $F$ acts trivially on $Z$, while if $(G,F)$ is of non-split type, 
then $F(\w) = -\w$.   

\para{2.3.}
Under the generalized Springer correspondence $\th^{\x}$ in (1.9.1), 
if $\x|_{Z_0} = 1$, $\th^{\x}$ is essentially the same as the generalized 
Springer correspondence for $SO_N$. 
Now assume that $\x|_{Z_0} \ne 1$, namely that $\x(\ve) = -1$. 
It is known by [L1, 14] that $\SM_{\x}$ is in bijection with integers $d \in \BZ$ 
such that $d \equiv N\pmod 4$ and $d(2d-1) < N$.  For such $d$, there exists 
a unique triple $(L, C_0, \SE_0) \in \SM_{\x}$ with $L$ of type 
$B_{(d-1)(2d+1)/2} + A_1 + A_1 + \cdots$ if $N$ is odd, and of type
$D_{d(2d-1)/2} + A_1 + A_1 + \cdots$ if $N$ is even, where the number of 
factors of $A_1$ is $(N - d(2d-1))/4$. 
In those cases, $\SW_L \simeq W_{(N - d(2d-1))/4}$, where $W_n$ 
is the Weyl group of type $B_n$. Thus $\th^{\x}$ gives a bijection 
\begin{equation*}
\tag{2.3.1}
\th^{\x} : \SN_{\x} \isom \bigsqcup_{\substack{d \in \BZ \\4|N - d }}
                            W_{(N - d(2d-1))/4}\wg.
\end{equation*}

\para{2.4.}
$\b : G \to \ol G$ gives a bijection $G\uni \isom \ol G\uni$, and induces
a bijection between the unipotent classes in $G$ and those in $\ol G$, which 
is compatible with $F$-action. 
For $u \in G\uni$, $\b(u)$ determines a partition of $N$, by taking the 
Jordan block of $\b(u)$, which we denote by $\la(u)$.  Here we write 
a partition $\la$ of $N$ as $\la = (\la_1, \dots, \la_k)$ with 
$0 < \la_1 \le \la_2 \le \cdots \le \la_k$ with $\sum_{i=1}^k\la_i = N$.
We also write $\la$ as $\la = (1^{m_1}, 2^{m_2}, \cdots)$.
Note that the set of unipotent classes in $O_N$ are in bijection with the set 
$\wt X_N$ of partitions 
$\la$ of $N$ such that $m_i$ is even for even $i$, by considering their Jordan types.
If all the $m_i$ are even, the class in $O_N$ splits into two classes in $SO_N$, 
otherwise any class $C$ in $O_N$ gives a single $SO_N$-class $C \cap SO_N$.  
\par
Put
\begin{equation*}
\tag{2.4.1}
X_N = \{ \la \in \wt X_N \mid m_i \le 1 \text{ for odd $i$} \}
\end{equation*}  
\par
For an integer $m \in \BZ$, put 
\begin{equation*}
\tag{2.4.2}
d(m) = \begin{cases}
         0  &\quad\text{ if $m$ : even }, \\
         1  &\quad\text{ if } m \equiv 1 \pmod 4,  \\
         -1 &\quad\text{ if } m \equiv -1 \pmod 4.       
        \end{cases}
\end{equation*}
For $\la \in X_N$, we define an integer $d(\la)$ by 
$d(\la) = \sum_id(\la_i)$. The following result was proved in [L1, Prop. 14.4], [LS, Cor. 4.10].

\begin{prop} 
Assume that $\x \in Z\wg$ is such that $\x(\ve) = -1$.  Then there exists a bijection 
between $\SN_{\x}$ and $X_N$ with the following properties;
\begin{enumerate}
\item
Let $(C, \SE) \in \SN_{\x}$ with $(u, \r) \lra (C,\SE)$.  
Then $\la(u) \in X_N$.  Under 
the correspondence (2.3.1), $(C,\SE)$  is mapped to the factor of $d$, 
where $d = d(\la)$ for $\la = \la(u)$.
\item
Let $I = \{ 1 \le i \le k \mid \la_i \text{ : odd} \}$.    
Define an integer $r$ by 
\begin{equation*}
r = \begin{cases}
            (|I|-1)/2   &\quad\text{ if $|I|$ is odd,}  \\
            (|I| -2)/2  &\quad\text{ if $|I|$ is even $> 0$, } \\
            0           &\quad\text{ if $|I| = 0$.}        
          \end{cases}
\end{equation*}
Then $\r$ is a unique irreducible representation of $A_G(u)$ of dimension $2^r$
such that $Z$ acts on $\r$ by $\x$.  
(More precisely, there exists a unique irreducible representation of 
dimension $2^r$ on which $\ve$ acts as $-1$ if $|I|$ is odd, and there exists
exactly two irreducible representations of dimension $2^r$ on 
which $\ve$ acts as $-1$ if $|I|$ is even.) 
\end{enumerate}
\end{prop}

\para{2.6.}
Take $(u,\r) \lra (C, \SE) \in \SN_{\x}$ as in Proposition 2.5.  
Following [L1, 14.2], we shall determine the structure of $A_G(u)$, and 
construct the representation $\r \in A_G(u)\wg$ explicitly.
Put $\b(u) = \ol u \in \ol G$.  Then we have $\b\iv(Z_{\ol G}(\ol u)) = Z_G(u)$.
$\b$ induces a natural surjective map $A_G(u) \to A_{\ol G}(\ol u)$. 
In the case where $\la \in X_N$, 
$\ve \notin Z_G^0(u)$, and $\b\iv(Z_{\ol G}(\ol u))$ consists of
two connected components.  Thus $A_G(u)$ is a central extension of $A_{\ol G}(\ol u)$
with a kernel of order 2.  
\par
This central extension can be described as follows. Let $I$ be the set as in 
Proposition 2.5. 
We can write $V = \bigoplus_{i \in I}V_i \oplus V'$, the orthogonal direct sum decomposition
into $\ol u$-stable subspaces, where $\dim V_i = \la_i$ for $i \in I$.
For each $i \in I$, we consider an orthonormal basis $v_1^i, \dots, v_h^i$ of $V_i$
(here we put $h = \la_i$).
Put $x_i = v^i_1v^i_2\cdots v^i_h \in C(V)$ (note that $h$ is odd).
Those $x_i$ satisfy the relation 
\begin{equation*}
\tag{2.6.1}
x_i^2 = \ve^{\la_i(\la_i -1)/2},  \quad 
x_ix_{i'} = \ve x_{i'}x_i.
\end{equation*}
In fact, it follows from the defining relations that
$(v^i_a)^2 = (v^i_a, v^i_a) = 1$, 
$v^i_av^i_{a'} + v^i_{a'}v^i_a = 2(v^i_a, v^i_{a'}) = 0$, namely 
$v^i_av^i_{a'} = - v^i_{a'}v^i_a$ for $a \ne a'$. 
Then we have, for $i \ne i'$, 
\begin{align*}
\tag{2.6.2}
x_i^2      &= (-1)^{h-1}(v^i_2\dots v^i_h)(v^i_2\dots v^i_h) 
                              = \cdots 
                              = (-1)^{h(h-1)/2}, \\
x_ix_{i'} &= (-1)^hv^{i'}_1(v^i_1\cdots v^i_h)(v^{i'}_2\cdots v^{i'}_{h'}) = 
         \cdots = (-1)^{hh'}(v^{i'}_1\cdots v^{i'}_{h'})(v^i_1\cdots v^i_h),  
\end{align*} 
where $h'= \la_{i'}$.  
Since $h, h'$ are odd, we obtain (2.6.1). 
\par
Let $\wh\vG$ be the subgroup of the group of units of $C(V)$ generated by 
$x_i$ ($i \in I$).   Let $\vG$ be the subgroup of $\wh\vG$ consisting of elements 
which are products of an even number of generators $x_i$. 
Then $\vG \subset Z_G(u)$, and the map $\vG \to Z_G(u)/Z_G^0(u)$ gives an isomorphism.
The central extension $Z_G(u)/Z_G^0(u) \to Z_{\ol G}(\ol u)/Z_{\ol G}^0(\ol u)$
can be described by $\vG \to \vG/\{ 1, \ve\}$. 
(Note that $\ve = x_ix_{i'}(x_{i'}x_{i})\iv \in \vG$.)     
\par
Let $C(V_I)$ be the Clifford algebra over $\Ql$ associated to the vector 
space $V_I \simeq \Ql^{|I|}$ and the quadratic form 
$\sum_{i \in I}(-1)^{\la_i(\la_i-1)/2}X_i^2$ with variables $X_i (i \in I)$, and 
$C^+(V_I)$ its $+$-part. 
Let $\Ql[\vG]$ be the group algebra of $\vG$ and $\Ql[\vG]/\lp 1 + \ve \rp$
the quotient algebra of $\Ql[\vG]$ by the two-sided ideal generated by
$1 + \ve$. Then $\Ql[\vG]/\lp 1 + \ve \rp \simeq C^+(V_I)$. 
It follows that there exists an algebra homomorphism $f: \Ql[\vG] \to C^+(V_I)$.
It is known (cf. [L1, 14.3]) that,$C^+(V_I)$ is a simple algebra in the case where $|I|$ is odd 
with $\dim C^+(V_I) = 2^{|I|-1}$, and is the direct sum of two simple components 
of the same dimension in the case where $|I|$ is even with $\dim C^+(V_I) = 2^{|I|-2}$.  
Thus if $|I|$ is odd, $f$ gives rise to an irreducible representation 
$\r : \vG \to GL(V_{\r})$, where $C^+(V_I) = \End (V_{\r})$, and if $|I|$ is 
even, it gives rise to two irreducible representations $\r, \r'$ such that 
$C^+(V_I) \simeq \End(V_{\r}) \oplus \End(V_{\r'})$.  
Here $\dim \r = 2^{(|I|-1)/2}$ if $|I|$ is odd, and 
$\dim \r = \dim \r'= 2^{(|I|-2)/2}$ if $|I|$ is even. 

\remark{2.7.}
We take $\x \in Z\wg$ such that $\x(\ve) = 1$. Note that in this case $\x$ is $F$-stable.
As mentioned in 2.3, the generalized Springer correspondence $\th^{\x}$ is essentially the same 
as  the case of $SO_N$.  By [S3, II], we already know that Theorem 1.5 holds for $SO_N$. 
Thus for $(C, \SE) \in \SN^F_{\x}$, there exists a split element 
$v^{\bullet} \in \b(C)^F$. 
By Proposition 2.5, the Jordan type $\la$ of $\b(C)$ is not contained in $X_N$. 
Since $\b$ gives a bijection $C^F \isom \b(C)^F$, one can find $u \in C^F$ such that
$\b(u) = v^{\bullet}$. In this case, it is known by [L1, 14.3] that 
$A_G(u) \simeq A_{\ol G}(v^{\bullet})$. This isomorphism is compatible with $F$, 
thus $A_G(u)$ is an elementary abelian two group with trivial $F$-action. 
We define a split element as $u = u^{\bullet} \in C^F$, then the $G^F$-class of 
$u^{\bullet}$ is uniquely determined by the split class in $\b(C)^F$.  
Theorem 1.5 holds for such $\x \in Z\wg$. Thus the verification of the theorem 
is reduced to the case where $\x(\ve) = -1$, and to the determination of 
the split class in $C^F$ for $C$ of Jordan type $\la \in X_N$.    

\remark{2.8.}
Assume that $F$ is of non-split type.  Then $N$ is even, 
and $Z$ is generated by $\ve$ and $\w$.  
We consider the set $\SN_{\x}$ in Proposition 2.5.
Then $\x(\ve) = -1$. 
Since $F$ is non-split, we have $F(\w) = -\w$ by (2.2.1).
Thus 
\begin{equation*}
\x(F(\w)) = \x(-\w) = \x(\ve)\x(\w) = -\x(\w).
\end{equation*} 
This shows that $\x$ is not $F$-stable, and so $\SN_{\x}^F = \emptyset$.
In the non-split case, we don't need to consider the situation as in 
Proposition 2.5, and so Theorem 1.5 automatically holds. 

\par\bigskip
\section{Proof of Theorem 1.5 -- the case of spin groups}

\para{3.1.}
By Remark 2.8, we may consider the case where 
$F$ is of split type. So hereafter we assume that $F$ is  split. 
Following [S3, II], we review the definition of split elements in $SO_N$.
Assume that $\la \in X_N$, and write it as 
$\la = (\la_1 \le \la_2 \le \cdots \le \la_k)$. 
\par\medskip
(a) \ Assume that $\la_j$ is odd, and put $h = \la_j$. 
We consider the vector space $V_j$ over $\Bk$ of $\dim V_j = h$
with basis $e_1^j, \dots, e_h^j$.
Put
\begin{equation*}
\tag{3.1.1}
\d_j = \d_j(\la) = (\la_j-1)/2 + j.
\end{equation*}
We define a non-degenerate symmetric bilinear form $f_j$ on $V_j$
by
\begin{equation*}
\tag{3.1.2}
f_j(e^j_{h-a+1}, e_a^j) = (-1)^{\d_j-a}, \qquad (1 \le a \le h)
\end{equation*} 
and $f_j = 0$ for all other pairs of bases. 
We define a nilpotent transformation $x_j$ on $V_j$ by 
$x_j(e^j_a) = e^j_{a-1}$ with the convention $e^j_0 = 0$.  
Then $x_j \in \Fg_j$, where $\Fg_j = \Lie SO(V_j, f_j)$. 
\par\medskip
(b) \ Assume that $\la_j = \la_{j+1}$ is even. Put $h = \la_j$.
We consider a vector space $V_j$ of dimension $2h = 2\la_j$ with basis
$e^j_1, \dots, e^j_h, e_1^{j+1}, \dots, e^{j+1}_h$. 
We define a non-degenerate symmetric bilinear form $f_j$ on $V_j$ by 
\begin{equation*}
\tag{3.1.3}
f_j(e^j_{h-a+1}, e_a^{j+1}) = (-1)^{a-1}, 
\qquad (1 \le a \le h) 
\end{equation*}
and $f_j = 0$ for all other pairs of bases. 
We define a nilpotent transformation $x_j$ on $V_j$ by 
$x_j(e^j_a) = e^j_{a-1}$, $x_j(e^{j+1}_a) = e^{j+1}_{a-1}$, 
with the convention $e^j_0 = e^{j+1}_0 = 0$. 
Then $x_j \in \Fg_j$, where $\Fg_j = \Lie SO(V_j, f_j)$. 
\par\medskip
Let $V = \bigoplus_jV_j$, where 
the sum is taken all $j$ such that $\la_j$ is odd, and a half of $j$ 
such that $\la_j$ is even (note that 
$\sharp\{j \mid \la_j = h\}$ is even for even $h$).  
We define a symmetric bilinear form 
$f = \bigoplus_jf_j$ on $V$.  Then $x = \sum_jx_j\in \Fg = \bigoplus_j\Fg_j$, 
where we can regard $\Fg = \Lie SO(V, f)$.    
If we assume that $\{ e^j_a \}$ is an $\BF_q$-basis of 
$V$, then $f$ gives an $F$-invariant symmetric bilinear form on $V$, which we simply
denote by $(\ ,\ )$, and also write $SO(V)$ as $\ol G$.  
$x$ determines the $\ol G^F$-conjugacy class of unipotent elements in $\ol G^F$, 
which is nothing but the split class in $\ol C^F$ for the unipotent class 
$\ol C$ in $\ol G$ with Jordan type $\la$.  We denote the split class by 
$\ol C^{\bullet} \subset \ol C^F$. 
We choose a split unipotent element $v = v^{\bullet} \in \ol C^{\bullet}$. 

\para{3.2.}
Let $C$ be the unipotent class in $G$ with Jordan type $\la \in X_N$. 
Then $\b$ induces a bijection $\b : C^F \isom \ol C^F$.  
We choose $u \in C^F$ such that $\b(u) = \ol u = v$. 
But in contrast to the case where $\la \notin X_N$ (see Remark 2.7), 
the set $\b\iv(\ol C^{\bullet})$ is not a single $G^F$-class, so 
the split class in $\ol C^F$ does not determine a unique $G^F$-class 
in $C^F$. We define the split classes in $C^F$ all the $G^F$-classes 
contained in $\b\iv(\ol C^{\bullet})$.  
In the following discussion, we shall show that those $G^F$-classes 
satisfy the property of the split class.  
\par
We consider the action of  $F$ on $A_G(u)$. 
$\b$ induces a surjective map $\wt\b: A_G(u) \to A_{\ol G}(\ol u)$ 
with $\Ker \wt\b = \{ 1, \ve\}$.
Here $\wt\b$ is $F$-equivariant, and we know that $F$ acts trivially 
on $A_{\ol G}(\ol u)$.
For any $x \in A_G(u)$, $\wt\b\iv(\wt\b(x)) = \{ x, \ve x\}$, 
and so $F(x) = x$ or $\ve x$.  Thus $F^2$ acts trivially on $A_G(u)$. 
\par
We determine the action of $F$ on $A_G(u)$ explicitly.
By the discussion in 2.6, $A_G(u)$ is constructed from  the orthonormal basis 
for $V_j$ with $\la_j$: odd.   So we need to convert the basis in $V_j$ to
an orthonormal basis.  
We consider the following situation.  Let $M$ be a vector space 
of $\dim M = h$ ($h$: odd) with basis $e_1, \dots, e_h$, endowed with 
non-degenerate symmetric bilinear form $(e_a, e_{h-a+1}) = (-1)^{c + a}$
for some fixed integer $c$.
Let $\z \in \Bk$ be such that $\z^2 = -1$.  Thus $F(\z) = \z$ if 
$q \equiv 1 \pmod 4$ and $F(\z) = -\z$ if $q \equiv -1 \pmod 4$.  
The following lemma can be checked by a direct computation. 

\begin{lem}  
Put $\g = (-1)^{c+1}$ and $h = 2m+1$. 
Assume that $F: M \to M$ is a Frobenius map with $F(e_i) = e_i$.
\begin{enumerate}
\item
We define elements $v_1', \dots, v_h'$ of $M$ by 
\begin{equation*}
\tag{3.3.1}
\begin{aligned}
v'_1 &= e_1 + \frac{\g}{2}e_{h},  &\quad v'_2 = e_2 -\frac{\g}{2}e_{h-1}, 
       &\quad v'_3 = e_3 + \frac{\g}{2}e_{h-2}, \dots  \\ 
v'_{h} &= e_1 - \frac{\g}{2}e_{h}, &\quad v'_{h-1} = e_2 + \frac{\g}{2} e_{h-1}, 
    &\quad v'_{h-2} = e_3 - \frac{\g}{2}e_{h-2}, \dots \\
v'_{m+1} &= e_{m+1}.
\end{aligned}
\end{equation*}
Then $(v'_1, v'_1) = (v'_2, v'_2) = \cdots = (v'_{m}, v'_{m}) = 1$,
     $(v'_{m+1}, v'_{m+1}) = (-1)^{m+1 + c}$, 
and $(v'_h, v'_h) = (v'_{h-1}, v'_{h-1}) = \cdots = (v'_{m+2}, v'_{m+2}) = -1$.
Moreover, $F(v'_a) = v'_a$ for $1 \le a \le h$.  

\item 
Define $v_1, \dots, v_h \in M$ by $v_1 = v'_1, v_2 = v'_2, \dots, v_m = v'_m$, 
$v_h = \z v'_h, v_{h-1} = \z v'_{h-1}, \dots v_{m+2} = \z v'_{m+2}$ and 
$v_{m+1} = v'_{m+1}$ (resp. $\z v'_{m+1}$) if $m+1 + c$ is even (resp. odd).
Then $v_1, \dots, v_h$ gives an orthonormal basis of $M$.  
Moreover, $F(v_i) = v_i$ for any $i$ if $q\equiv 1\pmod 4$.  
If $q \equiv -1\pmod 4$, then  
$F(v_1) = v_1, \dots, F(v_m) = v_m$, 
$F(v_h) = -v_h, F(v_{h-1}) = -v_{h-1}, \dots, F(v_{m + 2}) = -v_{m+2}$, 
and $F(v_{m+1}) = (-1)^{m +1 + c}v_{m+1}$.  
\end{enumerate}
\end{lem}

\para{3.4.}
Let $V_j$ be as in 3.1 for $\la_j = h$ : odd. 
Write $h = 2m + 1$.  Then $\d_j = m + j$. By applying Lemma 3.3 
for $M = V_j$ with $c = m + j$, one can find an orthonormal basis
$v_1^j, \dots, v_h^j$ of $V_j$. 
Put $x_j = v_1^j\cdots v_h^j \in C(V)$.  
The action of $F$ on the basis $\{ v_1^j, \dots, v_h^j\}$ is given 
as in the lemma.  In particular, $F(x_j) = x_j$ if $q\equiv 1\pmod 4$.
While if $q \equiv -1\pmod 4$, we have
\begin{equation*}
\tag{3.4.1}
F(x_j) = (-1)^{m + 1 + \d_j}(-1)^mx_j = (-1)^{(\la_j-1)/2 + 1 + j}x_j.
\end{equation*}

\par
By the discussion in 2.6, $A_G(u)$ is isomorphic to the subgroup 
of the group of units in $C^+(V)$ generated by an even number of products
$x_j$ such that $\la_j$ is odd. Thus we have the following lemma.

\begin{lem}  
If $q\equiv 1\pmod 4$, then $F$ acts trivially on $A_G(u)$.
While if $q\equiv -1\pmod 4$, then $F$ acts non-trivially on $A_G(u)$. 
\end{lem} 

\para{3.6.}
We apply the restriction formula (1.7.4) for the following case.
Let $M$ be the Levi subgroup of 
a parabolic subgroup $Q$ of $G$ such that $\b(M) \simeq SL_2 \times SO_{N-4}$. 
Take $(u,\r) \lra (C,\SE) \in (\SN_G)_{\x}$, and 
$(u',\r') \lra (C',\SE') \in (\SN_M)_{\x}$.  
Then $u'$ satisfies the properties  
\par\medskip\noindent
(i) \ The projection of $\b(u')$ on $SL_2$ is regular unipotent.
\\
(ii) \ The projection of $\b(u')$ on $SO_{N-4}$ corresponds 
to the unipotent elements of Jordan type in $X_{N-4}$ under the 
correspondence $(\SN_{\Spin_{N-4}})_{\x} \simeq X_{N-4}$. 
\par\medskip
Put $\b(M) = \ol M$, $\b(Q) = \ol Q$.  
Let $p : \ol Q \to \ol M$ be the natural projection. 
Put $\la = \la(u) \in X_N$ and $\la(u') = \la' \in X_{N-4}$. 
Assume that $\r \otimes \r'^*$ appears in $\ve_{u,u'}$. 
Then it is known by [LS, Lemma 4.5, Lemma 4.8], 
only the following 5 cases occur; 
there exists an integer $i$ such that one of the following conditions 
is satisfied.
\par\medskip
(I)  $\la_i$ is odd, $\la_i > \la_{i-1} + 4$, and 
\begin{equation*}
\la_j' = \begin{cases}
             \la_j - 4  &\quad\text{ if $j = i$, } \\
              \la_j     &\quad\text{ otherwise. }
         \end{cases}
\end{equation*}

(II) $\la_i = \la_{i+1} \ge \la_{i-1} + 2$ and 
\begin{equation*}
\la_j' = \begin{cases}
            \la_j-2 &\quad\text{ if $j = i, i+1$, }  \\
            \la_j   &\quad\text{ otherwise.} 
         \end{cases}
\end{equation*}

(III) $\la_i = \la_{i+1} \ge \la_{i-1}+ 4$ and
\begin{equation*}
\la_j' = \begin{cases}
            \la_j - 3  &\quad\text{ if $j = i$, }  \\
            \la_j - 1  &\quad\text{ if $j = i+1$, } \\
            \la_j      &\quad\text{ otherwise. }
         \end{cases}
\end{equation*}

(IV) $\la_{i+1} - 2 = \la_i \ge \la_{i-1} + 1$ and 
\begin{equation*}
\la'_j = \begin{cases}
           \la_j - 1 &\quad\text{ if $j = i$, }  \\
           \la_j - 3 &\quad\text{ if $j = i+1$, } \\
           \la_j     &\quad\text{ otherwise. }  
         \end{cases}
\end{equation*}

(V) $\la_{i+2} = \la_{i+1} = \la_i + 1$ and
\begin{equation*}
\la'_j = \begin{cases}
            \la_j - 1 &\quad\text{ if $j = i, i+2$, } \\
            \la_j - 2 &\quad\text{ if $j = i+1$, }  \\
            \la_j     &\quad\text{ otherwise. } 
         \end{cases} 
\end{equation*}

\para{3.7.}
Assume that $(u, \r) \lra (C, \SE) \in (\SN_G)_{\x}$.  We choose 
a split element $u \in C^F$, and assume that $F$ acts non-trivially on 
$A_G(u)$.  Since $F^2$ acts trivially on $A_G(u)$ (see 3.2), 
the order of $\tau$ is equal to 2.  There exist two extensions $\wt\r, \wt\r'$ 
of $\r$ to $\wt A_G(u)$.  
We need to fix an appropriate extension $\wt\r$ of $\r$,
called the split extension of $\r$.
Assume given $(u',\r') \in (\SN_M)_{\x}$ as in 3.6, and choose 
a split element $u' \in C'^F$.  
Recall that $(u_0, \r_0) \lra (C_0, \SE_0)$ for the cuspidal pair in 
$L$.  We assume that 
the split extensions $\wt\r_0$ of $\r_0$ and $\wt\r'$ of $\r'$ are
already determined.  We consider $\wt\ve_{u, u'}$ and its subspace 
$\r \otimes \r'^*$.  
$\tau'$ acts on $\wt\r'$ via $T'$. 
If we fix an extension $\wt\r$ of $\r$, where $\tau$ acts as $T$ on $\wt\r$, 
then $\tau_*$ act as $T_* = T\otimes T'$ with $T_*^2 = 1$. $\tau$ acts as $-T$ for $\wt\r'$.
Since $\r \otimes \r'^*$ is 
$F$-stable, the action of $F$ on $\r\otimes \r'^*$ coincides with a scalar 
times $T_*$.  If we can show that $F^2$ acts trivially on $\r \otimes \r'^*$, 
then by replacing $\wt\r$ by $\wt\r'$
if necessary, we can find an extension $\wt\r$ of $\r$ such that 
$F$ acts as $T_*$ on $\wt{\r \otimes \r'^*}$.  
We call such $\wt\r$ the split extension of $\r$. 
By (1.8.2), we have 

\begin{lem} 
Assume that $F$ acts non-trivially on $A_G(u)$.  If 
$F^2$ acts trivially on the subspace $\r\otimes \r'^*$ of $\ve_{u,u'}$, 
then for the split extension $\wt\r$ 
of $\r$, we have $\g = 1$. 
\end{lem}  

\para{3.9.}
Recall that
$Y_{u,u'} = \{ g \in G \mid g\iv ug \in u'U_Q\}$, and 
$C'$ the unipotent class in $M$ containing $u'$.  
Put 
\begin{equation*}
\tag{3.9.1}
\CQ_{u,C'} = \{ gQ \in G/Q \mid g\iv ug \in C'U_Q\}.
\end{equation*} 
We have a natural map $q: Y_{u,u'} \to \CQ_{u,C'}$. $Z_G(u)$ acts on 
$\CQ_{u, C'}$ from the left, and the map $q$ is $Z_G(u)$-equivariant.
$q$ induces an isomorphism 
\begin{equation*}
\tag{3.9.2}
Y_{u,u'}/Z_M(u')U_Q \isom \CQ_{u,C'}.
\end{equation*}

If $u \in C^F, u' \in C'^F$, then $Y_{u,u'}, \CQ_{u,C'}$ are defined over 
$\BF_q$, and the map $q$ is $F$-equivariant. 
Put $\ol Q = \b(Q)$.  We have $\b(Y_{u,u'}) = \ol Y_{\ol u,\ol u'}$, which 
is the variety defined with respect to $\ol G$ and $\ol M$.  Then 
$\b : Y_{u,u'} \to \ol Y_{\ol u,\ol u'}$ gives a double covering. 
The variety $\ol\CQ_{\ol u, \ol C'}$ is defined similarly to $\CQ_{u,C'}$,
and a formula analogous to (3.9.2) holds for $\ol Y_{\ol u,\ol u'}$.
We show a lemma

\begin{lem}  
Assume that $F$ acts trivially on $A_G(u)$ and on $A_M(u')$. 
Assume that $\ol\CQ_{\ol u,\ol C'}$ has finitely many $Z_{\ol G}(\ol u)$-orbits, 
and each $Z_{\ol G}(\ol u)$-orbit contains an element $\ol y \ol Q$
such that $\ol y \in \ol Y^F_{\ol u,\ol u'}$. Then 
$F$ acts trivially on $\ve_{u,u'}$.  
\end{lem}

\begin{proof}
Let $\ol y\ol Q \in \ol\CQ_{\ol u,\ol C'}$ be such that $\ol y \in \ol Y^F_{\ol u,\ol u'}$.
Then there exists $y \in G^F$ such that $yQ = \ol y\ol Q$ and that 
$y\iv uy \in u_1'U_Q$ for $u_1' \in C'^F$. This implies that 
$\ol y\iv \ol u \ol y \in \ol u_1'U_{\ol Q}$, and $\ol u_1'$ is $\ol M^F$-conjugate 
to $\ol u'$.  Hence $\ol u_1'$ is a split element in $\ol C'$, and so 
$u'_1$ is a split element in $C'$. 
We have $y \in Y^F_{u, u'_1}$, and $\ol y \in \ol Y^F_{\ol u, \ol u_1'}$.  
By replacing $u'$ by $u_1'$ if necessary, we may assume that $y \in Y^F_{u,u'}$. 
Let $R$ be the $Z_{\ol G}(\ol u)$-orbit of $\ol y \ol Q$ in $\ol\CQ_{\ol u,\ol C'}$.  
Then 
$q\iv R = Z_{\ol G}(\ol u)\ol y Z_{\ol M}(\ol u')U_{\ol Q}$.
Since $\ve = -1 \in Z_G(u)$, we can write as 
$\b\iv(q\iv R) = Z_G(u)yZ_M(u')U_Q$, with $y \in Y^F_{u,u'}$.
Then 
\begin{equation*}
\b\iv(q\iv R) = \bigsqcup_{\a \in A_G(u), \b \in A_M(u')}
                 Z_G^0(u)x_{\a}yx'_{\b}Z_M^0(u')U_Q,
\end{equation*}
where $x_{\a}, x'_{\b}$ are representatives of $\a, \b$ in $Z_G(u)$, $Z_M(u')$ 
respectively.  
Hence the irreducible components of $\b\iv( q\iv R)$ are given by 
$Z_G^0(u)x_{\a}yx'_{\b}Z_M(u')$.
Since $F$ acts trivially on $A_G(u), A_M(u')$, all those irreducible components 
are $F$-stable.  Thus $F$ acts trivially on $\ve_{u,u'}$.
\end{proof}

\begin{prop}  
Lt $u \in C, u' \in C'$ with $\la(u) = \la$, $\la'= \la(u')$.
Assume that the type of $(\la, \la')$ is given as in {\rm (I), (II), (IV)} 
or {\rm (V)} 
in 3.6. Let $u \in C^F$ be a split element such that $F$ acts trivially on 
$A_G(u)$.  Then there exists 
a split element $u' \in C'^F$ such that $F$ acts trivially on $\ve_{u,u'}$. 
(For the case {\rm (III)}, see Remark 3.13 below.) 
\end{prop}

\begin{proof}
We can identify $\CQ = G/Q$ with the variety $\SF$ of all 2-dimensional 
totally isotropic subspaces of $V$.
Then $\CQ_{u,C'}$ can be identified with the subvariety $\SF_{u,C'}$ of $\SF$
consisting of all $E \in \SF$ such that 
\begin{enumerate}
\item 
$E$ is $\ol u$-stable.
\item
$\ol u|_E \ne 1$.
\item 
The Jordan type of $\ol u|_{E^{\perp}/E}$ is $\la'$. 
\end{enumerate}  

Let $u \in C^F$ be a split element such that $F$ acts trivially on $A_G(u)$.  
In each case of (I), (II), (IV) or (V) in 3.6, we shall find a split element 
$u' \in C'^F$ satisfying the assumption in Lemma 3.10.  Then the proposition follows 
from the lemma. 
\par\medskip\noindent
{\bf Case I.} 
\ Let $V_i$ be the subspace of $V$ corresponding to $\la_i$, 
and $e_1, \dots, e_h$ 
be the $F$-stable basis of $V_i$ (with $h = \la_i$) such that 
$xe_a = e_{a-1}$ for $x \in \Fg^F$ as in 3.1 (a).
Put $E = \lp e_1, e_2\rp$.  It is easy to see that 
$\SF_{u,C'} = \{ E \}$. 
$SL(E) \times SO(E^{\perp}/E)$ is $\ol G^F$-conjugate to $\ol M$. 
Thus there exists $\ol u' \in \ol M^F$ such that 
$\ol u'$ is $\ol G^F$-conjugate to $(\ol u|_E, \ol u|_{E^{\perp}/E})$.
Then $u' = \b\iv(\ol u')$ is a split element in $C'^F$. 
Since $A_G(u) \simeq A_M(u')$ with $F$-action, $F$ acts trivially on 
$A_M(u')$.  
One can choose $\ol y \in \ol G^F$ such that $\ol y\ol Q$ corresponds to $E$.  
Then $\ol y \in \ol Y^F_{\ol u,\ol u'}$. 
Thus the assumption of Lemma 3.10 is satisfied. 
\par\medskip\noindent
{\bf Case II. }
Let $V_i$ be as in 3.1 (b).  We write the basis 
$e^i_1, \dots, e^i_h, e_1^{i+1}, \dots, e_h^{i+1}$ of $V_i$ as
$e_1, \dots, e_h, f_1, \dots, f_h$, where $h = \la_i = \la_{i+1}$.  
Under the decomposition 
$V = \bigoplus_kV_k$ in 3.1, 
we define a subspace $M_i$ of $V$ as the direct sum of $V_k$ such that 
$\la_k = h$.  Let $x \in \Fg$ be as in 3.1.  Take 
$E \in \SF_{u,C'}$.  Then $\dim  (E \cap \Ker x) = 1$.
We choose $0 \ne v \in E \cap \Ker x$.  $v$ is $G_1$-conjugate to $e_1 \in V_i$, 
where $G_1$ is a subgroup of $Z_{\ol G}(\ol u)$ which is isomorphic to $Sp(M_i)$.
So assume that $e_1 \in E \cap \Ker x$.  Then the condition on $\la'$ implies that
$E \subset V_i$.  We see that $E$ is of the form 
$E = \lp e_1, e_2 + \a f_1\rp$ with $\a \in \Bk$. 
It follows that 
\begin{equation*}
\tag{3.11.1}
\SF_{u,C'} = \bigcup_{g \in G_1}g\{ \lp e_1, e_2 + \a f_1 \rp \mid \a \in \Bk \}.
\end{equation*}  
We choose $E = \lp e_1, e_2 \rp$.  Then $E  \in \SF^F_{u,C'}$.
There exists $\ol u' \in \ol M^F$ such that $\ol u'$ is $\ol G^F$-conjugate to
$(\ol u|_E, \ol u|_{E^{\perp}/E})$.  Then $u' = \b\iv(\ol u') \in C'^F$
is a split element.  Since $A_G(u) \simeq A_M(u')$ with $F$-action, $F$ acts trivially 
on $A_M(u')$.    
It is easy to see that $\lp e_1, e_2 + \a f_1\rp$ is $Z_{\ol G}(\ol u)$-conjugate 
to $E$.  Thus $\SF_{u,C'}$ coincides with the $Z_{\ol G}(\ol u)$-orbit of $E$. 
There exists $\ol y \in \ol Y_{\ol u,\ol u'}^F$ such that $\ol y\ol Q$ corresponds to $E$.
The assumption of Lemma 3.10 is satisfied. 
\par\medskip\noindent
{\bf Case IV.} \
Put $M_i = V_i \oplus V_{i+1}$, where $\dim V_{i+1} = h $ and $\dim V_i = h-2$.
Write the bases of $V_{i+1}, V_i$ in 3.1 as 
$e_1, \dots, e_h$ for $V_{i+1}$, and $e_1', \dots, e'_{h-2}$ for $V_i$.
Take $E \subset \SF_{u,C'}$.  
Then $E$ must be in $M_i$.  
Take $0 \ne v \in M_i \cap \Ker x$.  The condition for $\la'$
implies that $v = e_1$ up to scalar. 
Thus we may assume that $v = e_1$. 
$E$ can be written as $E = \lp e_1, e_2 + \a e_1'\rp$ with $\a \in \Bk$. 
The condition that $x|_{E^{\perp}/E}$ has type $\la'$ is given by 
\begin{equation*}
\tag{3.11.2}
( e_2 + \a e_1', e_{h-1} + \a e'_{h-2}) = 0. 
\end{equation*}
This implies that $(e_2, e_{h-1}) + \a^2(e'_1, e'_{h-2}) = 0$. 
Under the notation of (3.1.1), we have $\d_i = \d_{i+1}$. 
Thus by (3.1.2), $(e_2, e_{h-1}) = -(e'_1, e'_{h-2})$, and so $\a = \pm 1$.
We have
\begin{equation*}
\tag{3.11.3}
\SF_{u,C'} = \{ E = \lp e_1, e_2 + e_1'\rp, E' = \lp e_1, e_2 -e_1'\rp \ \}.
\end{equation*}
In particular, both of $E, E'$ are $F$-stable. As in Case I, 
we can define $u' \in C'^F$ by using $(\ol u|_E, \ol u|_{E^{\perp}/E})$. 
Then $u'$ is a split element. 
$A_M(u')$ is obtained from $A_G(u)$ by removing $x_i$ and $x_{i+1}$, 
under the notation in 2.6. Thus $F$ acts trivially on $A_M(u')$.  
As in the previous cases, we can find $\ol y \in \ol Y_{\ol u,\ol u'}^F$ 
such that $\ol y\ol Q$ corresponds to $E$.
If we choose $\ol y' \in \ol G^F$ such that $\ol y'\ol Q$ corresponds to $E'$, 
then ${\ol y'}\iv \ol u \ol y' \in \ol u''U_{\ol Q}$ 
for some $\ol u'' \in \ol C'^F$.  But then $\ol u''$ is also a split element in 
$\ol C''^F$, and $\ol u''$ is $\ol M^F$-conjugate to $\ol u'$.  
By replacing $\ol y'$ appropriately if necessary, we find 
$\ol y' \in \ol Y^F_{\ol u,\ol u'}$ such that $\ol y'\ol Q$ corresponds to $E'$. 
Thus the assumption of Lemma 3.10 is satisfied. 
\par\medskip\noindent
{\bf Case V. } \
Put $h = \la_{i+1} = \la_{i+2}$, with $h$ : even.
Let $M_{i+1}$ be the subspace of $V$ which is the direct sum of $V_k$ such that
$\la_k = h$. Then $V_{i+1} \subset M_{i+1}$.  We write the basis 
of $V_{i+1}$ given in 3.1 (b) as $e_1, \dots, e_h, f_1, \dots, f_h$, 
and write the basis of $V_i$ in 3.1 (a) as $e_1', \dots, e'_{h-1}$. 
Take $E \in \SF_{u,C'}$.   Then $E \subset M_{i+1}\oplus V_i$. 
Take $0 \ne v \in E \cap \Ker x$.  Similarly to Case II, 
by taking the $G_1$-conjugate, we may assume that $v \in V_{i+1} \oplus V_i$, 
where $G_1$ is the subgroup of $Z_{\ol G}(\ol u)$ which is isomoprhic to $Sp(M_{i+1})$.
Then $E \subset V_{i+1}\oplus V_i$. 
It follows from the condition for $\la'$ that $v$ must be in $V_{i+1}$.
Then under the conjugation by $Z_{\ol G}(\ol u)$, we may assume that $v = e_1$. 
$E$ is given as $E = \lp e_1, e_2 + z\rp$ with 
$z \in  (V_{i+1} \oplus V_i)\cap \Ker x$.   
So we write $z = \a f_1 + \b e'_1$. 
Then the condition $x|_{E^{\perp}/E}$ has type $\la'$ is given by 
\begin{equation*}
\tag{3.11.4}
(e_2 + \a f_1 + \b e_1', e_h + \a f_{h-1} + \b e'_{h-1}) = 0. 
\end{equation*}  
This implies that
$\a (e_2, f_{h-1}) + \a (f_1, e_h) + \b^2(e'_1, e'_{h-1}) = 0$. 
Since $(e_2, f_{h-1}) = (f_1, e_h) = -1$ by 3.1 (b), we have
\begin{equation*}
\tag{3.11.5}
\a = \frac{\b^2(e_1', e'_{h-1})}{2}.
\end{equation*}
Thus for any $\b \in \Bk$, there exists a unique 
$E_{\b} = \lp e_1, e_2 + z\rp \in \SF_{u,C'}$ given as above.  In particular, if 
we put $E = \lp e_1, e_2\rp$, then $E \in \SF_{u,C'}$ and $E_{\b}$ is
$Z_{\ol G}(\ol u)$-conjugate to $E$. Thus $Z_{\ol G}(\ol u)$ 
acts transitively on $\ol\CQ_{\ol u,\ol C'}$.  
Here $E \in \SF^F_{u,C'}$, and $(\ol u|_{E}, \ol u|_{E^{\perp}/E})$ 
determines $u' \in C'^F$.  
There exists $\ol y \in \ol Y^F_{\ol u,\ol u'}$ such that $\ol y\ol Q$ corresponds to $E$.
Note our process to obtain $u'$ from $u$ corresponds to the procedure,
\begin{equation*}
(\la_i, \la_{i+1}, \la_{i+2}) \mapsto (\la_i, \la_{i+1}-2, \la_{i+1}-2)
    \mapsto (\la_{i+1}-2, \la_{i+2} -2, \la_i),
\end{equation*}
where the second one is a rearrangement of the order of the parts.
Here the first one is $(\la_i, \la_{i+1}, \la_{i+2}) = (h-1, h, h)$ 
and the last one is $(\la'_i, \la'_{i+1}, \la'_{i+2}) = (h-2, h-2, h-1)$. 
We have $\d_i = (\la_i -2)/2 + i$ by (3.1.1).  
Then $\d'_{i+2} = (\la'_{i+2} - 1)/2 + (i+2) \equiv \d_i \pmod 2$. 
It follows that $u' \in C'^F$ is a split element. 
Since $A_G(u) \simeq A_M(u')$ with $F$-action, $F$ acts trivially on $A_M(u')$.
Thus the assumption of Lemma 3.10 is satisfied.
\end{proof}

\remark{3.12.}
In the statement in Proposition 3.11, 
the case (III) is not used in the later discussion, so we omitted it. 
Actually in this case, the structure of $\SF_{u,C'}$ is more complicated, 
and the splitness of $u' \in C'^F$ is not deduced from the split element 
$u \in C^F$.  
\par\medskip
We are now ready to prove Theorem 1.5 in the case of spin groups.
Note that split elements are defined in 3.2, and split extensions are 
defined in 3.7.

\begin{thm} 
Assume that $G = \Spin_N$. 
For each $(C, \SE) \in \SN_G^F$, 
let $u^{\bullet} \in C^F$ be a split element, and assume that 
$(C,\SE) \lra (u^{\bullet}, \r)$. 
For each $(L, C_0, \SE_0) \in \SM_G^F$, let $u_0^{\bullet} \in C_0^F$ be a 
split element as above (applied for $L$), 
and assume that $(C_0, \SE_0) \lra (u_0^{\bullet}, \r_0)$. 
Let $\wt\r_0 \in \wt A_L(u_0^{\bullet})\wg$ be an extension of $\r_0$, 
and for each $(C, \SE)$ belonging to the series $(L, C_0, \SE_0)$, 
let $\wt\r \in \wt A_G(u^{\bullet})\wg$ be the split extension of $\r$ associated to 
$\wt\r_0$.   
Then $\g(u_0^{\bullet}, \wt\r_0, u^{\bullet}, \wt\r) = 1$.
\end{thm}

\begin{proof}
By Proposition 2.5, we may assume that $(C, \SE) \in (\SN_G)^F_{\x}$ 
for $\x \in Z\wg$ such that $\x(\ve) = -1$. Then the Jordan type 
of $\ol C$ is $\la \in X_N$. 
By Remark 2.8, $\SN_{\x}^F = \emptyset$ if $F$ is of non-split type. 
Thus we may assume that $F$ is of split type. 
We define a split element $u = u^{\bullet} \in C^F$ as in 3.1 and 3.2.
Assume that $N \ge 7$, and consider an $F$-stable Levi subgroup 
$M$ of $G$ such that $\b(M) \simeq SL_2 \times SO_{N-4}$. 
Take $(C',\SE') \in (\SN_M)_{\x}$, and assume that $(C, \SE) \lra (u, \r)$, 
and $(C',\SE') \lra (u',\r')$, where $\la = \la(u), \la' = \la(u')$. 
Then the condition that $\r \otimes \r'^*$ 
appears in $\ve_{u,u'}$ is exactly that the pair $(\la, \la')$ satisfies 
the condition (I) $\sim$ (V) in 3.6.  
Now for a given $u \in C$, one can find 
$u' \in C'$ such that $(\la, \la')$ satisfies one of the conditions 
(I), (II), (IV) or (V). (Note that the case (III) is not necessary, since 
if the condition (III) is applicable for $\la_i$, 
then already the condition (II) can be applied for $\la_i$.)
We choose a split element $u' \in C'^F$.  Then by Proposition 3.11, 
$F$ acts trivially on $\ve_{u,u'}$, and also $F$ acts trivially on $A_M(u')$.  
Then by (1.8.1), $\s_{\r, \r'}: M_{\r,\r'} \to M_{\r,\r'}$ is 
the scalar multiplication by $q^{-(\dim C - \dim C')/2 + \dim U_Q}$.
By induction, we may assume that $\s_{(u',\r')}$ acts on $V_{(u',\r')}$
as a scalar multiplication by $q^{(a_0'+ r')/2}$, where $a_0',r'$ 
are defined similarly to $a_0, r$ by replacing $G$ by $M$.       
It follows that $\s_{(u,\r)}$ acts on $V_{(u,\r)}$ as a scalar multiplication 
by $q^{(a_0 + r)/2}$, and we conclude that $\g = 1$ by (1.8.2).  
\par
It remains to consider the case where $N \le 6$.
If $N = 5$ or 6, $(u, \r)$ is a cuspidal pair by [LS, Cor.4.9]. Thus we have 
$\g = 1$. So we assume that $N = 4$ or 5.   
First consider the case where $N = 5$.  In this case, 
$X_5 = \{ (5), (122)\}$, 
and $\SN_{\x}$ corresponds to $(L, C_0, \SE_0) \in \SM_{\x}$, where 
$L$ is a Levi subgroup of $G = \Spin_5$ such that $L \simeq SL_2$, and 
for $(u_0, \r_0) \lra (C_0, \SE_0)$, $u_0$ is a regular unipotent element of $L$, 
and $\r_0$ is a non-trivial character of $A_L(u_0) = Z_L \simeq \BZ/2\BZ$. 
Thus if we regard $X_1 = \{ (1)\}$ as the set corresponding to $u_0$, 
we have a bijection $X_1 \simeq (\SN_L)_{\x}$, and (1.8.1) holds for this situation. 
In particular the verification of the theorem for 
the case $N = 5$ is reduced to the case $N = 1$, namely the case where $G = SL_2$.
In this case, the assertion is already known by [S3, I, \S 3], and actually it is 
easy to check directly. 
If $N = 4$, then $X_4 = \{ (22), (13)\}$, and $\SN_{\x}$ corresponds to 
$(L, C_0, \SE_0) \in \SM_{\x}$, where $L$ is a Levi subgroup of 
$G = \Spin_4$ such that $L \simeq SL_2$, ($u_0, \r_0$) is 
the same as above.  In this case, we can regard $X_0 = \{ (-) \}$ 
as the set corresponding to
$u_0$, we have a bijection $X_0 \simeq  (\SN_L)_{\x}$, and (1.8.1) 
holds. Thus the theorem holds also for the case where $N = 4$. 
We have proved the theorem in the case where $F$ acts trivially on $Z_G(u)$.
\par
Next we assume that $F$ acts non-trivially on $A_G(u)$.  
But then $F^2$ acts trivially 
on it.  Thus the theorem holds if we replace $F$ by $F^2$. 
By Lemma 3.8, we can find a split extension $\wt\r$ of $\r$ for each split element 
$u \in C^F$.  Thus the theorem holds for this case.  
The theorem is now proved. 
\end{proof}
\par\bigskip

\section{Proof of Theorem 1.5 -- the case of special linear groups}

\para{4.1}
In this section, we assume that $G = SL_n$, and $\ch \Bk$ is arbitrary.
Let $p$ be the characteristic exponent of $\Bk$, namely, 
write $p = \ch \Bk$ if $\ch \Bk > 0$, and $p = 1$ if $\ch \Bk = 0$.  
Let $F = \s F_0$ be a twisted Frobenius map on $G$, where $F_0$ is 
the standard Frobenius map, and $\s$ is the graph automorphism 
defined by $\s(g) = \dw_0 {}^tg\iv \dw_0\iv$, where $\dw_0 \in G$ 
is the permutation matrix corresponding to the longest element 
$w_0 \in W = S_n$. 
\par
Let $Z$ be the center of $G$.  Then $Z = \{ \a I_n \mid \a \in \Bk, \a^n = 1\}$
is the cyclic group of order $n'$, where $n'$ is the largest divisor of $n$
which is prime to $p$. 
Let $u = u_{\la} \in G$ be a unipotent element with Jordan type $\la$ of $n$, 
and we denote by $C_{\la}$ the unipotent class in $G$ containing $u_{\la}$. 
We write $\la = (\la_1, \la_2, \dots \la_r) = (1^{m_1}, 2^{m_2}, \dots)$,
where $0 < \la_1 \le \la_2 \le \cdots \le \la_r$, 
and put $I = \{ i \in \BZ_{\ge 1} \mid m_i \ne 0\}$.  
Let  
$\{ v_{k,j} \mid 1 \le k \le r, 1 \le j \le \la_k \}$
be the Jordan basis of $u$ such that $(u -1)v_{k,j} = v_{k, j-1}$ with the 
convention $v_{k,0} = 0$. For $i \in I$, we define a subspace 
$V_i$ of $V$ by $V_i = \lp v_{k, j} \mid \la_k = i \rp$. 
We have a decomposition 
$V = \bigoplus_{i \in I}V_i$ into $u$-stable subspaces, 
where the Jordan type of $u|_{V_i}$ is $(i^{m_i})$.   
Let $V_i^0$ be the subspace of $V_i$ generated by 
$\{ v_{k,i} \mid \la_k = i\}$.   
(Note that the notation $V_i$ here is different from 
that in 3.1.) 
The Levi part $M_u$ of $Z_G(u)$ is isomorphic to 
the subgroup of $\prod_i GL(V_i^0)$ given by 
\begin{equation*}
\tag{4.1.1}
M _u = \{ g = (g_i) \in \prod_{i \in I}GL(V_i^0) \mid \prod_{i\in I} (\det g_i)^i = 1 \}.
\end{equation*} 
Let $i'$ be the $p'$-part of $i$.  Then the condition  in the right hand side of 
(4.1.1) is equivalent to the condition that $\prod_i (\det g_i)^{i'} = 1$. 
Let $n'_{\la}$ be the greatest common dividor of $\{ i' \mid i \in I\}$, and write
$i' = i''n'_{\la}$.  Thus we have 
\begin{equation*}
\prod_i(\det g_i)^{i'} = \biggl(\prod_i(\det g_i)^{i''}\biggr)^{n'_{\la}}. 
\end{equation*}
It follows that the connected component $M_u^0$ of $M_u$ is given by  
\begin{equation*}
\tag{4.1.2}
M_u^0 = \{ g = (g_i) \in \prod_{i \in I}GL(V_i^0) \mid \prod_{i \in I}(\det g_i)^{i''} = 1\},
\end{equation*}
and $A_G(u) \simeq M_u/M_u^0$ is a cyclic group of order $n'_{\la}$.
Note that $n'_{\la}$ coincides with the greatest common divisor of 
$\la_1, \dots, \la_r, n'$. 

\para{4.2.}
Let $\x \in Z\wg$ and consider $A_G(u)\wg_{\x}$. 
We have a natural map $\BZ/n'\BZ \to  \BZ/n'_{\la}\BZ$, 
and $\r \in A_G(u)\wg_{\x}$ is equivalent to the condition that 
$\x$ is the lift of $\r$.
If $\x \in Z\wg$ is of order $d$, then $\r$ is of order $d$ also, 
and $d$ is a divisor of $n'_{\la}$. For a given $\x$, there exists 
such a $\r$ uniquely. 
Thus we have 
\begin{equation*}
\tag{4.2.1}
|A_G(u)_{\x}\wg|  = \begin{cases}
                       1    &\quad\text{ if $d$ divides each $\la_i$, } \\
                       0    &\quad\text{ otherwise. }
                    \end{cases}
\end{equation*}

\para{4.3.}
It is known by [L1] that in the case where $G = SL_n$, 
$(C, \SE)$ is a cuspidal pair 
for $G$ if and  only if $C$ is the regular unipotent class, and 
$\r \in A_G(u)\wg = Z\wg$ has order $n$ for $(u,\r) \lra (C, \SE)$.
Thus for $\x \in Z\wg$ of order $n$, there exists exactly one 
cuspidal pair $(C_u, \SE_{\r})$ pertaining $\x$, where $u$ : regular unipotent, 
$\r \in A_G(u)\wg_{\x}$.
\par
The generalized Springer correspondence is described as follows.
Fix $\x \in Z\wg$. Then there exists a unique Levi subgroup  
$L$ which has a cuspidal pair $(C_0, \SE_0) \in (\SN_L)_{\x}$, 
and this cuspidal pair is also unique.  Assume that $\x$ is of order $d$.
Then $L = S(GL_d \times GL_d \times \cdots \times GL_d)$ with $n/d$-factors. 
$u_0 \in C_0$ is a regular unipotent element in $L$, and 
$\r_0 \in A_L(u_0)\wg$ corresponding to $\SE_0$ is given 
as follows. We have a natural map $Z_{SL_d} \to Z_L(u_0)/Z^0_L(u_0)$, which 
gives an isomorphism $Z_{SL_d} \simeq A_L(u_0)$.  Thus $A_L(u_0)$ is the cyclic
group of order $d'$, and $(u_0, \r_0)$ is uniquely determined from the cuspidal 
pair of $SL_d$ pertaining $\x$.     
In this case, $N_G(L)/L \simeq S_{n/d}$, the cyclic group of degree $n/d$. 
Thus the generalized Springer correspondence pertaining $\x$ is given as follows;
\begin{equation*}
\tag{4.3.1}
\SN_{\x} = \{ (u, \r) \mid \text{ $\la_i$ : divisible by $d$}, 
                              \r \in A_G(u)\wg_{\x} \} 
         \lra S_{n/d}\wg,  
\end{equation*}
where $(u,\r) \lra (C_{\la},\SE)$. 
(Here $\r$ is determined uniquely from $u$ by (4.2.1).)
\par
By [LS, Prop. 5.2], this correspondence is given by 
$u_{\la} \mapsto E_{\mu} \in S_{n/d}\wg$, where 
$\mu = \la/d = (\la_1/d, \la_2/d, \dots ) \in \SP_{n/d}$, and 
$E_{\mu}$ is the irreducible representation of $S_{n/d}$ corresponding 
to $\mu$.

\para{4.4.}
We consider $(L, C_0, \SE_0) \in \SM_G$, where $L$ is of type 
$A_{d-1} + A_{d-1} + \cdots$.  Then $L$ is $F$-stable, 
and the regular unipotent class $C$ is $F$-stable.  The center $Z$ of $G$
is the set of scalar matrices, hence the action of $F_0$ and of $F$ on $Z$
are given by, for $z \in Z$, $F_0(z) = z^q$, $F(z) = z^{-q}$.  
Thus the actions of $F_0, F$ on $Z\wg$ are given by $\x \mapsto \x^q$, 
$\x \mapsto \x^{-q}$.  Hence 
$\x$ is $F_0$-stable (resp. $F$-stable) if and only if $d|q-1$ (resp. $d|q+1$). 
Now take $u_0 \in C_0^F$.  Then $F$ acts on $A_L(u_0)$. 
We have $A_L(u_0) \simeq Z_{SL_d}$, and the action of $F$ is compatible with 
this isomorphism.  It follows that if $\x \in Z\wg$ is $F$-stable, 
then the corresponding $\SE_0$ is $F$-stable. 
We have $N_G(L)/L \simeq S_{n/d}$, and the action of $F$ on $N_G(L)/L$ corresponds 
to the $\s$-action on $S_{n/d}$, namely the conjugation action by $w_0$.
\par
We shall apply the restriction formula (1.8.1) for $(L, C_0, \SE_0)$.   
Let $Q$ be the $F$-stable parabolic subgroup of $G$, and $M$ its $F$-stable Levi
subgroup such that $M \simeq S(GL_d \times GL_{n-2d} \times GL_d)$.   
Take $(C, \SE) \in \SN_G$ and $(C',\SE') \in \SN_M$, and $u \in C, u' \in C'$.
Here $u$ is of type $\la$ and $u'$ is of type $\la'$. 
We consider $Y_{u,u'}$, and $\CQ_{u,C'}$ as in 3.9.  Put $x = u-1$.  
Then $\CQ_{u,C'}$ can be identified with
\begin{align*}
\tag{4.4.1}
\SF_{u, \la'} = \{ (W \subset W ') &\mid \dim W = d, \dim W ' = n-d,  
       W, W ' \text{: $x$-stable, } \\
     &x|_{W '/W} \text{ : type $\la'$},
          x|_{W}, x_{V/W '} \text{: regular nilpotent} \}.
\end{align*}

\para{4.5.}
Since $d|\la_i$ for any $i$, we put $\mu = \la/d \in \SP_{n/d}$, and similarly 
$\mu' = \la'/d \in \SP_{n/d -2}$.  
Then $\mu'$ is obtained from $\mu$ by removing two nodes from the corresponding 
Young diagram. 
\par
For the pair $(\la, \la')$, there are three possibilities;
we write $\mu : \mu_1 \le \mu_2 \le \cdots, \mu': \mu_1'\le \mu'_2 \le \cdots$.
\par\medskip
(I) \ $\mu_{i-1} + 2  \le \mu_i$, and $\mu'_i = \mu_i -2$.
\par
\hspace*{-1mm}(II) \ $\mu_{i-1} + 1 \le \mu_i = \mu_{i+1}$, and $\mu'_i = \mu'_{i+1} = \mu_i-1$.
\par
\hspace{-2mm}(III) \ $\mu_{i-1} + 1 \le \mu_i$ and $\mu_{j-1} + 1 \le \mu_j$ for $i \ne j$, 
and $\mu'_i = \mu_i -1, \mu'_j = \mu_j -1$.  
\par\medskip
Let $E_{\mu} \in S_{n/d}\wg$ and $E_{\mu'} \in S\wg_{n/d -2}$. 
Then 
\begin{equation*}
\tag{4.5.1}
\lp E_{\mu}|_{S_{n/d -2}}, E_{\mu'} \rp_{S_{n/d -2}} =
                \begin{cases}
                   1   &\quad\text{ for the case (I) or (II), }  \\
                   2   &\quad\text{ for the case (III).} 
                \end{cases}
\end{equation*}
We fix $\mu$, and consider the possibility for the choice of $\mu'$.
If there exists $i$ such that $\mu_{i-1} \le \mu_i -2$, we can apply (I).
If there exists $i$ such that $\mu_{i-1}+1 \le \mu_i = \mu_{i+1}$, 
we can apply (II).  Thus for (III), we have only to consider the case
where $\mu_i = \mu_{i-1}+1$ for any $i$. 
But in this case, we can choose $\mu_i' = \mu_i-1, \mu'_{i+1} = \mu_{i+1}-1$, 
namely we can choose $i \ne j$ in (III) as $j = i+1$. Moreover, all the rows
have multiplicity one.

\begin{lem}   
Let $\la = d\mu \in \SP_n, \la' = d\mu' \in \SP_{n-2d}$ be as in 4.5
\begin{enumerate}
\item 
Assume that $(\la, \la')$ is in the case {\rm (I)} or {\rm (II)} in 4.5.  
Then $Z_G(u)$ acts transitively on $\SF_{u,\la'}$.
\item
Assume that $(\la, \la')$ is in the case {\rm (III)} in 4.5. 
Let $\nu \in \SP_{n-d}$ be the partition obtained from $\la$ 
by replacing $\la_i$ by $\nu_i = \la_i-d$, and $\nu' \in \SP_{n-d}$ 
by replacing $\la_j$ by $\nu'_j = \la_j-d$.  Put
\begin{align*}
\SF^{\nu}_{u, \la'} &= \{ (W \subset W ') \in \SF_{u,\la'}  
                          \mid u|_{V/W} \text{ : type $\nu$}\}, \\
\SF^{\nu'}_{u,\la''} &= \{ (W \subset W') \in \SF_{u,\la'}
                           \mid u|_{V/W} \text{ : type $\nu'$}\}.
\end{align*} 
Then $\SF_{u,\la'} = \SF^{\nu}_{u,\la'} \sqcup \SF^{\nu'}_{u, \la'}$, 
and $Z_G(u)$ acts transitively on $\SF^{\nu}_{u,\la'}$ and on $\SF^{\nu'}_{u,\la'}$. 
In particular, $\SF_{u,\la'}$ has two $Z_G(u)$-orbits. 
\end{enumerate}
\end{lem}
 
\begin{proof}
Assume that $(\la, \la')$ is in the case (I) or (II).
Let $\mu'' \in \SP_{n/d -1}$ be the partition obtained from $\mu$
by replacing $\mu_i$ by $\mu_i'' = \mu_i-1$.  Put $\la'' = d\mu''$.   
Put
\begin{align*}
\SF^1_{u,\la''} = \{ W \subset &V \mid \dim W = d, W 
      \text{ : $x$-stable}, \\
        &\text{$x|_W$ : regular nilpotent, $x|_{V/W}$ : type $\la''$} \}. 
\end{align*}
We have a surjective map $\vf : \SF_{u,\la'} \to \SF^1_{u,\la''}$, 
$(W  \subset W ') \mapsto W$. Then $\vf$ is $Z_G(u)$-equivariant.  
It is clear that $Z_G(u)$ acts transitively on $\SF^1_{u,\la''}$ 
(note that we may replace $Z_G(u)$ by $Z_{GL_n}(u)$, and it is enough 
to consider the case where $d = 1$). 
We choose $W \in \SF^1_{u,\la''}$, and put $\ol V = V/W$.  $u$ acts on 
$\ol V$, and put $\ol u = u|_{\ol V}$. Then the fibre $\vf\iv(W)$ 
is isomorphic to the variety of subspaces $\ol W' \subset \ol V$ such that
$\ol W'$ is $\ol u$-stable, $\ol u|_{\ol W'}$ is of type $\la'$,
$\ol u|_{\ol V/\ol W'}$ is regular unipotent. 
But this variety is the same type as $\SF^1_{u, \la''}$, if we 
consider the dual space $\ol V^*$ of $\ol V$ and a $\ol u$-stable subspace 
$\ol W^{\perp} \subset \ol V^*$.  In particular $Z_{\ol G}(\ol u)$ acts transitively 
on $\vf\iv(W)$, where $\ol G = SL(\ol V)$. 
Thus $Z_G(u)$ acts transitively on $\SF_{u,\la'}$.
\par
The case $(\la, \la')$ is in the case (III) is dealt similarly. 
\end{proof}
\para{4.7.}
Let $V$ be an $n$-dimensional vector space over $\Bk$, with 
standard basis $e_1, \dots, e_n$.  We denote by $V_0$ the 
$\BF_{q^2}$-subspace of $V$ generated by $\{ e_i\}$. 
We consider the sesquilinear form $\Psi(v,w) : V_0 \times V_0 \to \BF_{q^2}$,
i.e., $\Psi(v,w) = {}^tv A \ol w$, where we write 
${}^tv = (v_1, \dots, v_n) \in \Bk^n$, $\ol w = {}^t(w_1^q, \dots, w_n^q)$, 
and $A$ is the Hermitian matrix, i.e., ${}^tA = \ol A$. 
Then $A$ is unitary congruent to the identity matrix ([G, Th. 10.3]), 
i.e., there exists $B \in GL_n(\BF_{q^2})$ such that $A = {}^tB\cdot \ol B$. 
It follows that any sesquilinear form $\Psi$ is unitary congruent to the
form $(v, w) = \sum_{i=1}^nv_iw_{n-i}^q$.    
\par
Let $\la = (\la_1, \dots, \la_r) \in \SP_n$, 
and $C = C_{\la}$ the corresponding unipotent class. 
Let $V_0$ be the $\BF_{q^2}$-subspace of $V$ generated by 
the Jordan basis $\{ v_{k,j}\}$ of $u \in C^{F^2}$. 
We define a sesquilinear form $(\ ,\ )$ on $V_0$ by
\begin{equation*}
\tag{4.7.1}
  (v_{k, j}, v_{k', j'}) = \begin{cases}
                        (-1)^{j + a_k} 
                          &\quad\text{ if $k = k', j + j' = \la_k$, } \\
                        0     &\quad\text{ otherwise. }
                                  \end{cases},
\end{equation*} 
where $a_k = \pm 1$. 
Then $u \in C^{F^2}$ leaves the form $(\ ,\ )$ invariant, and so $u \in C^F$. 
We call $u$ a split 
element in $C^F$.  Thus split elements are determined up to 
$GL_n^F$-conjugates.  
\par
Let $(\la, \la')$ be as in 4.5.  We choose $u \in C_{\la}^F$ a split element.
Then $\SF_{u, \la'}$ and $\SF^{\nu}_{u,\la'}, \SF^{\nu'}_{u,\la'}$
in Lemma 4.6 are $F$-stable.  We show the following lemma. 

\begin{lem}  
Let $u \in C_{\la}^F$ be a split element.
\begin{enumerate}
\item
Assume that $(\la, \la')$ is in the case {\rm (I)} or {\rm (II)} of 4.5.
Then there exists $E \in \SF^F_{u,\la'}$,  a split element 
$u' \in C'^F$ and $y \in Y_{u,u'}^F$ such that $yQ$ corresponds to 
$E$ under the identification $\CQ_{u,C'} \simeq \SF_{u,\la'}$.  
\item
Assume that $(\la, \la')$ is in the case {\rm (III)} with $j = i+1$. 
Then there exist
a split element $u' \in C'^F$, 
$E \in (\SF^{\nu}_{u,\la'})^F$, $E' \in (\SF^{\nu'}_{u,\la'})^F$
and $y, y' \in Y_{u,u'}^F$ such that 
$yQ$ (resp. $y'Q$) corresponds to $E$ (resp. $E'$). 
\end{enumerate}
\end{lem}

\begin{proof}
First consider the case (I) in 4.5.  
Let $W_0$ be the subspace of $V_0$ spanned by $v_{i,1}, \dots, v_{i,d}$.   
Then $W_0$ is $u$-stable, $u|_{W_0}$ is regular unipotent. 
We consider the flag $(W _0\subset W_0^{\perp})$ in $V_0$.  
Then $u|_{W_0^{\perp}/W_0}$ is of type $\la'$ and $u|_{V_0/W_0^{\perp}}$
is regular unipotent. 
Thus if we put $W = \Bk\otimes_{\BF_{q^2}}W_0, W' = \Bk \otimes_{\BF_{q^2}}W_0^{\perp}$, 
the flag $E = (W \subset W')$ is contained in $\SF_{u,\la'}$, and is $F$-stable. 
Moreover $u|_{W_0^{\perp}/W_0}$ is a split element of type $\la'$ on $W_0^{\perp}/W_0$.
Here $M$ is $GL_n^F$-conjugate to $S(GL(W) \times GL(W'/W) \times GL(V/W'))$.  
Thus there exists a split element 
$u' \in C'^F$ such that $u'$ is $GL_n^F$-conjugate to 
$(u|_W, u|_{W'/W}, u|_{V/W'})$. 
This implies that there exists $y \in Y_{u,u'}^F$ such that $yQ$ corresponds to 
$E$ under the isomorphism $\CQ_{u,C'} \simeq \SF_{u,\la'}$.  
\par
Next consider the case (II) in 4.5. 
Put $\la_i = \la_{i+1} = h$, and $x = u-1$. 
Let $U_0$ be the subspace of $V_0$ spanned by $v_{i,h}, v_{i+1, h}$.
We define a form $\lp\ ,\ \rp$ on $U_0$ by
$\lp v, w \rp = (v, x^{h-1}w)$.  
Then $\lp v,w\rp$ is a non-degenerate sesquilinear form on $U_0$. 
We choose $0 \ne v \in U_0$ such that $\lp v, v\rp = 0$. 
Let $W_0$ be the subspace of $V_0$ spanned by $x^{h-1}v, x^{h-2}v, \dots, x^{h-d}v$.
Then $W_0$ is a $u$-stable subspace of $V_0$ with dimension $d$, and 
$u|_{W_0}$ is regular unipotent. We have $u|_{W_0^{\perp}/W_0}$ is of type $\la'$, 
and $u|_{V_0/W_0^{\perp}}$ is regular unipotent.  Thus if we define the corresponding
subspace $W , W'$ in $V$ as before,  
we see that $E = (W \subset W') \in \SF_{u,\la'}^F$. 
Since $u|_{W_0^{\perp}/W_0}$ is a split element on $W_0^{\perp}/W_0$, 
one can find a split element $u' \in C'^F$, and $y \in Y_{u,u'}^F$ such that 
$yQ$ corresponds to $E$, similarly to the previous case.
\par
Finally consider the case (III) of 4.5.
Put $h = \la_i$. Write  
$v_{i,1}, \dots, v_{i,h}$ as $e_1, \dots, e_h$, and 
write $v_{i+1,1}, \dots, v_{i+1, h+d}$ as $e_1', \dots, e'_{h + d}$. 
Let $W_0$ be the subspace of $V_0$ spanned by 
$\a e_1 + e_1', \dots, \a e_d + e'_d$ for $\a \in \BF_{q^2}$. 
Then $W_0$ is $u$-stable, 
$u|_{W_0}$ is regular unipotent, $u|_{W_0^{\perp}/W_0}$ is of type $\la'$, 
and $u|_{V_0/W_0^{\perp}}$ is regular unipotent.    
Here note that $u|_{V_0/W_0}$ is of type $\nu$ (resp. $\nu'$) 
if $\a \ne 0$ (resp. $\a = 0$).  Thus if we define $W, W'$ 
the corresponding subspaces of $V$ for $\a \ne 0$, we see that 
$E = (W \subset W') \in (\SF_{u,\la'}^{\nu})^F$, while if $W, W'$
are defined from $W_0$ with $\a = 0$, then 
$E' = (W \subset W') \in (\SF^{\nu'}_{u,\la'})^F$.   
Hence the assertion can be proved similarly. The lemma is proved. 
\end{proof}
 
\para{4.9.}
Assume that $(\la, \la')$ is in the case (I) or (II) of 4.5.
We choose a split element $u \in C_{\la}^F$.  Then  by 
Lemma 4.8, there exists a split element $u' \in C'^F$ and 
$y \in Y_{u,u'}^F$.  By Lemma 4.6, $Z_G(u)$ acts transitively on 
$\CQ_{u,C'}$.  Thus $I_0 = Z^0_G(u)yZ_M^0(u')U_Q$ is an $F$-stable 
irreducible component of $Y_{u,u'}$, and all other irreducible components
of $Y_{u,u'}$ are obtained by the action of $A_G(u) \times A_M(u')$. 
Let $A_0'$ be the stabilizer of $I_0$ in $A_G(u) \times A_M(u')$, and 
$A_0$ be the quotient group $(A_G(u) \times A_M(u'))/A'_0$ (note that 
$A_G(u) \times A_M(u')$ is abelian).   Then $X_{u,u'}$ is in bijection with 
$A_0$, and $\ve_{u,u'}$ is isomorphic to $\Ql[A_0]$ as $A_0$-modules. 
Note that $F$ acts naturally on $A_0$, and the isomorphism 
$\ve_{u,u'} \simeq \Ql[A_0]$ is compatible with $F$-actions.
Let $\r_1$ be the irreducible character of $A_0$ corresponding to
$\r \otimes \r'^*$.  We assume that $\r$ and $\r'$ are $F$-stable.
Then $\r_1$ is $F$-stable. Now the subspace $\r \otimes \r'^*$ of $\ve_{u,u'}$ 
coincides with the subspace of $\Ql[A_0]$ on which $A_0$ acts via $\r_1$. 
Hence this is a one-dimensional subspace spanned by 
\begin{equation*}
z_{\r_1} = \sum_{a \in A_0}\r_1(a)\iv a \in \Ql[A_0].
\end{equation*}
Since $\r_1$ is $F$-stable, we see that $F$ acts trivially on 
$\r \otimes \r'^* \subset \ve_{u,u'}$.  

\para{4.10.}
Next assume that $(\la, \la')$ is in the case (III).
We follow the notation in the proof of Lemma 4.8. 
By Lemma 4.6 (ii), we have a partition 
$\SF_{u,\la'} = \SF_{u,\la'}^{\nu} \sqcup \SF_{u,\la'}^{\nu'}$. 
For $E = (W \subset W') \in \SF_{u, \la'}$, we have $\dim (W \cap \Ker x)= 1$,
and $W \cap \Ker x \subset \lp e_1, e_1'\rp$.  
Thus we have a map $\SF_{u,\la'} \to \BP(\lp e_1, e_1'\rp),(W \subset W') \mapsto W \cap \Ker x$.  
Then $W \cap \Ker x = \lp v \rp$, where $v = \a e_1 + e_1'$
for some $\a \in \BF_{q^2}$ (note that $E \notin \SF_{u,\la'}$ if $v = e_1$).  
Here $E \in \SF^{\nu}_{u,\la'}$ if $\a \ne 0$, and 
$E \in \SF^{\nu'}_{u,\la'}$ if $\a = 0$. It follows that $\SF^{\nu}_{u,\la'}$ is 
open in $\SF_{u,\la'}$.  
Let $\CQ^{\nu}_{u,C'}$ (resp. $\CQ^{\nu'}_{u,C'}$) be the variety corresponding 
to $\SF^{\nu}_{u,\la'}$ (resp. $\SF^{\nu'}_{u,\la'}$),  
and put $Y_{u,u'}^{\nu} = q\iv(\CQ_{u,C'}^{\nu})$ 
(resp. $Y_{u,u'}^{\nu'} = q\iv(\CQ_{u,C'}^{\nu'})$), where $q : Y_{u,u'} \to \CQ_{u,C'}$ 
is the map defined as in 3.10. We have 
$Y_{u,u'} = Y_{u,u'}^{\nu} \sqcup Y_{u,u'}^{\nu'}$, and $Y_{u,u'}^{\nu}$ is open 
in $Y_{u,u'}$.  Hence $\dim Y_{u,u'}^{\nu} = d_{u,u'}$.   
$Y_{u,u'}^{\nu}$ is stable under the action of $Z_G(u) \times Z_M(u')U_Q$.  
If we put $X_{u,u'}^{\nu}$ the set of irreducible components of dimension $d_{u,u'}$, 
then $X_{u,u'}^{\nu}$ is an $A_G(u) \times A_M(u')$-stable (non-empty) subset of $X_{u,u'}$. 
Thus we have a subrepresentation $\ve_{u,u'}^{\nu}$ of $\ve_{u,u'}$. 
We consider the representation $\r \otimes \r'^*$ appearing in $\ve_{u,u'}$.
By 4.5 (III), the multiplicity of $\r\otimes \r'^*$ in $\ve_{u,u'}$ is
equal to 2. Since $A_G(u) \times A_M(u')$ acts transitively on the set $X_{u,u'}^{\nu}$, 
a similar argument as in 4.9 shows that $\r \otimes \r'^*$ appears 
in $\ve_{u,u'}^{\nu}$ with multiplicity one.  In particular, 
$X_{u,u'}^{\nu} \ne X_{u,u'}$ and we see that $\dim Y_{u,u'}^{\nu'} = d_{u,u'}$.
If we define $X_{u,u'}^{\nu'}$ as the set of irreducible components of dimension $d_{u,u'}$, 
we have $X_{u,u'} = X_{u,u'}^{\nu} \sqcup X_{u,u'}^{\nu'}$. 
Thus $\ve_{u,u'}$ is decomposed as 
$\ve_{u,u'} = \ve_{u,u'}^{\nu}\oplus \ve_{u,u'}^{\nu'}$, where $\r\otimes \r'^*$ 
appears with multiplicity one in either of them. 
By applying the argument in 4.9, 
we see that $F$ acts trivially on $\r \otimes \r'^* \subset \ve_{u,u'}^{\nu}$, 
and on $\r \otimes \r'^* \subset \ve_{u,u'}^{\nu'}$. 
Hence $F$ acts trivially on the (two-dimensional) 
isotypic subspace of $\r\otimes \r'^*$ in $\ve_{u,u'}$.    
\par
Summing up the discussion in 4.9 and 4.10, we have the following.

\begin{prop}   
Let $(\la, \la')$ be as in 4.5.  For a split element $u \in C_{\la}^F$, 
there exists a split element $u' \in C_{\la'}^F$ such that 
$F$ acts trivially on the isotypic subspace of $\r\otimes \r'^*$ in $\ve_{u,u'}$. 
\end{prop}

\para{4.12.}
For each $(C, \SE) \in \SN_G^F$, we choose a split element 
$u = u^{\bullet}\in C^F$.  Let $\r \in A_G(u)\wg$ be an $F$-stable irreducible 
representation of $A_G(u)$ corresponding to $\SE$. Note that
$A_G(u)$ is abelian, and $\r$ is one-dimensional. 
As a split extension $\wt\r$ of $\r$, we choose a trivial extension such that
$\tau$ acts trivially on the representation space of $\r$. 
\par
Let $\x \in Z\wg$ be an $F$-stable irreducible character of order $d$. 
We consider the correspondence $\SN_{\x} \simeq S_{n/d}\wg$ as in 
(4.3.1) associated to $(L, C_0, \SE_0) \in \SM_{\x}$. 
We are now ready to prove Theorem 1.5 for $SL_n$.
Note that split elements are defined in 4.7.

\begin{thm}  
Assume that $G = SL_n$, with $F$: non-split type. 
For each $(C, \SE) \in \SN_{\x}^F$, let $u^{\bullet} \in C^F$ be a split element, 
and assume that $(C,\SE) \lra (u^{\bullet},\r)$. 
For each $(L, C_0, \SE_0) \in \SM_G^F$, let $u_0^{\bullet} \in C_0^F$
be a split element as above (applied for $L$), and assume that 
$(C_0,\SE_0) \lra (u_0^{\bullet}, \r_0)$.  
Then for each $(C,\SE)$ belonging to the series $(L,C_0,\SE_0)$, 
we have $\g(u_0^{\bullet}, \wt\r_0, u^{\bullet}, \wt\r) = 1$, where 
$\wt\r_0 \in \wt A_L(u_0^{\bullet})\wg$ (resp. $\wt\r \in \wt A_G(u^{\bullet})\wg$)
is the split extension of $\r_0$ (resp. $\r$).
\end{thm}   

\begin{proof}
Take $(u, \r) \lra (C, \SE) \in (\SN_G)_{\x}^F$ with $C = C_{\la}$.
We consider an $F$-stable Levi subgroup $M$ of $G$ as in 4.4. 
Then there exists $\la' \in \SP_{n-2d}$ such that $(\la, \la')$ 
is in the case (I) $\sim$ (III) in 4.5.
Take $(u',\r') \lra (C',\SE') \in (\SN_M)_{\x}^F$. 
By the restriction formula, $\r \otimes \r'^*$ appears in 
the representation $\ve_{u,u'}$.  We assume that 
$u \in C^F$ is a split element. Then by Proposition 4.11 
there exists a split element $u' \in C'^F$ such that $F$ acts trivially 
on the isotypic subspace of $\r \otimes \r'^*$ in $\ve_{u,u'}$.  
Then by (1.8.1), $\s_{\r,\r'}: M_{\r,\r'} \to M_{\r,\r'}$ is a scalar 
multiplication by $q^{-(\dim C - \dim C')/2 + \dim U_Q}$.  
By induction on $n$, we 
may assume that $q^{-(a_0'+ r')/2}\s_{(u',\r')}$ makes the irreducible 
$S_{n/d -2}$ module $V_{(u',\r')}$ to
the preferred extension to $\wt S_{n/d - 2}$.  
Then $q^{-(a_0 + r)/2}\s_{(u,\r)}$ makes the irreducible $S_{n/d}$-module 
$V_{(u,\r)}$ to the preferred extension to $\wt S_{n/d}$.   
This implies, by (1.8.2), that $\g(u,\wt\r, u_0, \wt\r_0) = 1$. 
It remains to consider the case where $n = d$, 
$(u,\r) \lra (C_{\la}, \SE) \in \SN_{\x}$.
But in this case $\la = (d)$ and $u$ is regular unipotent,   
$\r = \x \in A_G(u)\wg = Z\wg$.  Then $(C, \SE)$ is a cuspidal pair, 
and we have $\g = 1$ automatically. The theorem is proved. 
\end{proof}

\remark{4.14.}
In the proof of Theorem 4.13, we have assumed that 
$F$ is of non-split type.  But the argument there works well for the
split case, under a suitable modification, and  
it gives a proof of Theorem 1.5 for the split case.  Of course, 
in the split case, one can use $M \simeq S(GL_d \times GL_{n-d})$, 
which makes the argument simpler.  In any case, if we use the restriction 
formula (1.8.1), the discussion  turns out to be much simpler 
compared to that used in [S3, I].

\par

\par\vspace{1.0cm}
\noindent
T. Shoji \\
School of Mathematical Sciences, Tongji University \\ 
1239 Siping Road, Shanghai 200092, P.R. China  \\
E-mail: \verb|shoji@tongji.edu.cn|


\bigskip  
\begin{thebibliography}{[Spa1]}

\bibitem[BS]{BS} W.M. Beynon and N. Spaltenstein, Green functions 
of finite Chevalley groups of type $E_n$ ($n = 6,7,8$), J. Algebra {\bf 88}
(1984), 584--614.
\bibitem[G]{G} L.C. Grove; ``Classical groups and geometric algebra'', 
Graduate Studies in Math.{\bf 39},  Providence R.I. AMS, 2002.  
\bibitem[L1]{L1}G. Lusztig, Intersection cohomology complexes on 
a reductive group,
Invent. Math. {\bf 75} (1984), 205--272.
\bibitem[L2] {L2}G. Lusztig, Character sheaves, I, Adv. in
Math. {\bf 56} (1985),
193--237, II, Adv. in Math. {\bf 57} (1985), 226--265, III, Adv. in
Math. {\bf 57}
(1985), 266--315, IV, Adv. in Math. {\bf 59} (1986), 1--63, V, Adv. in Math.
{\bf 61} (1986), 103--155.
\bibitem[LS]{LS} G. Lusztig and N. Spaltenstein, On the generalized
Springer correspondence for classical groups, Advanced Studies in Pure
Math. Vol. {\bf 6} (1985), 289--316.
\bibitem[S1]{S1} T. Shoji, On the Green polynomials of Chevalley groups
of type $F_4$, Comm. in Alg. {\bf 10} (1982), 505-543. 
\bibitem[S2] {S1}T. Shoji, On the Green polynomials of classical 
groups, Invent. Math. {\bf 74}, (1983), 237--267.
\bibitem[S3]{S3}T. Shoji, Generalized Green functions and
unipotent classes for finite reductive groups, I, 
Nagoya Math. J. {\bf 184} (2006), 155--198, II, Nagoya Math. J. {\bf 188} (2007), 133--170. 

\end{thebibliography}
\end{document}